\documentclass[10pt, reqno]{amsart}
\usepackage{amssymb, amsthm, amsmath, amsfonts}
\usepackage{graphics}
\usepackage{hyperref}
\usepackage{bbm}
\usepackage[all]{xy}
\usepackage{enumerate}
\usepackage[mathscr]{eucal}
\usepackage{thmtools}
\usepackage{xcolor}
\usepackage{extarrows}
\usepackage{ytableau}
\usepackage{enumitem}
\usepackage{zhlipsum}
\usepackage{cite}
\usepackage[english]{babel}
\usepackage{tikz-cd}[row sep=large, column sep=large]
\theoremstyle{plain}
\newtheorem{theorem}{Theorem}[section]
\newtheorem{lemma}[theorem]{Lemma}
\newtheorem{proposition}[theorem]{Proposition}

\newtheorem{corollary}[theorem]{Corollary}
\numberwithin{equation}{section}

\theoremstyle{definition}

\newtheorem{definition}[theorem]{Definition}
\newtheorem{example}[theorem]{Example}
\newtheorem{remark}[theorem]{Remark}

\newcommand{\C}{\mathscr C}
\newcommand{\D}{\mathscr D}
\newcommand{\Cp}{\mathscr C^{+}}
\newcommand{\Cm}{\mathscr C^{-}}
\newcommand{\I}{\mathfrak I}
\newcommand{\Bal}{\operatorname{Bal}}
\newcommand{\End}{\operatorname{End}}
\newcommand{\Aut}{\operatorname{Aut}}
\newcommand{\Hom}{\operatorname{Hom}}
\newcommand{\im}{\operatorname{im}}
\newcommand{\Inj}{\operatorname{Inj}}
\newcommand{\AF}{\operatorname{AF}}
\newcommand{\SAF}{\operatorname{SAF}}
\newcommand{\GL}{\operatorname{GL}}
\newcommand{\AGL}{\operatorname{AGL}}
\newcommand{\Sing}{\operatorname{Sing}}
\newcommand{\rank}{\operatorname{rank}}

\newcommand{\id}{\operatorname{id}}
\newcommand{\Ob}{\operatorname{Ob}}
\newcommand{\Mod}{{\textrm{-}\mathrm{Mod}}}
\newcommand{\VA}{{\mathrm{VA}}}
\newcommand{\SVA}{{\mathrm{SVA}}}

\title{Normalization graphs and Dold-Kan equivalences for generalized Reedy categories}
\author{Liping Li}
\address{School of Mathematics and Statistics, Hunan Normal University, Changsha 410081, China.}
\email{lipingli@hunnu.edu.cn}
\thanks{The author is partly supported by NSFC Grant No. 12571037.}


\begin{document}

\begin{abstract}
We introduce a combinatorial criterion for the projectivity of standard modules over generalized Reedy categories, formulated in terms of a finite bipartite graph whose normalized weightings determine normalization idempotents and hence projective splittings of standard modules. We apply this framework to recover several classical normalization constructions and Dold--Kan--type equivalences, and to obtain new Dold--Kan--type equivalences for categories of finite-dimensional affine, semilinear, and semiaffine spaces over a finite field. These equivalences also yield explicit descriptions of uniformly continuous representations of the corresponding infinite transformation monoids. Together with our previous work \cite{LiHom} on Hom-spaces between standard modules, the results provide a combinatorial machinery, based mainly on graph theory and linear algebra, for establishing Dold--Kan--type equivalences for generalized Reedy categories.

\end{abstract}

\maketitle

\section{Introduction}

\subsection{Motivation}

The classical Dold--Kan correspondence identifies simplicial modules with nonnegatively graded chain complexes \cite{Dold1958,Kan1958}, and is one of the fundamental examples in which a category of representations defined by rather complicated combinatorial operators admits a much simpler algebraic description. Many analogues have subsequently appeared for categories carrying related combinatorial structures. Among them are the normalization theory for cyclic modules developed by Dwyer and Kan \cite{DwyerKan1,DwyerKan2}, the equivalence between cubical abelian groups with connections and chain complexes by Brown--Higgins \cite{BrownHiggins}, the Dold--Kan correspondence for dendroidal abelian groups by Guti\'{e}rrez--Lukacs--Weiss \cite{GLW}, Pirashvili's equivalence for Segal's category $\Gamma$ \cite{Pirashvili}, the classification of $\mathrm{FI}^{\sharp}$-modules in terms of symmetric-group representations by Church--Ellenberg--Farb \cite{CEF}, and Kuhn's equivalence for finite-dimensional vector spaces over a finite field in non-describing characteristic \cite{Kuhn}. More general Morita-theoretic frameworks encompassing many such examples were developed by Helmstutler \cite{Helm}, Lack and Street \cite{LS, St}, S{\l}omi\'{n}ska \cite{Slo}, and, in the homotopical setting, Kaygun-Kaya \cite{KaygunKaya} and Walde \cite{Walde}.

Concrete proofs of these results are usually strongly adapted to the particular combinatorics of the category under consideration. For instance,
\begin{itemize}
\item for the simplex category one uses the simplicial identities to construct the classical normalization operator;

\item for $\Gamma$, normalization is expressed through cross-effects and an inclusion--exclusion procedure;

\item for $\VA_q$, Kuhn's argument ultimately relies on the idempotent in the full matrix semigroup algebra constructed by Kov\'{a}cs \cite{Kovacs}.
\end{itemize}
However, examining these examples suggests a common representation-theoretic mechanism described as follows: starting with the representable projectives, one first performs a normalization procedure to obtain projective summands, shows that these normalized modules form another family of projective generators, computes the Hom-spaces between them, and finally applies Morita theory for rings with several objects to deduce the desired equivalence of module categories.

A further observation is particularly relevant to generalized Reedy categories introduced by Berger--Moerdijk \cite{BergerMoerdijk2011}. In most of the examples above, the normalized projectives are precisely the \emph{standard modules} associated with the Reedy structure. Thus, in this setting, the problem of establishing a Dold--Kan--type equivalence for a generalized Reedy category $\C$ is reduced to two essentially independent questions:
\begin{enumerate}
\item when is every standard module of $\C$ projective?
\item what are the Hom-spaces between these standard modules?
\end{enumerate}
A satisfactory answer to the first question identifies the standard modules as normalized projectives, while a transparent answer to the second identifies explicitly the linear category to which $\C$ is Morita equivalent.

The second part of this machinery was developed in our previous paper \cite{LiHom}. There we introduced elementary combinatorial constructions, based mainly on graphs and linear algebra, which make it possible to describe homomorphisms between standard modules for a broad class of generalized Reedy categories. The purpose of the present paper is to complete the first part of the picture using tools of the same elementary nature. Instead of searching directly for category-specific normalization idempotents, we associate to every object a finite bipartite graph, called the \emph{normalization graph}. Projectivity of standard modules is translated into the existence of a normalized weighting on this graph with prescribed boundary values. Once such a weighting exists, the corresponding normalization operator idempotent is obtained automatically. In this way the usual direction of argument is reversed as follows:
\[
\text{normalization graphs}
   \ \Longrightarrow\
\text{normalized weightings}
   \ \Longrightarrow\
\text{normalization idempotents}.
\]
This is parallel to the Hom-space machinery developed in our previous paper \cite{LiHom}, which may be summarized schematically as
\[
\text{fiber graphs}
   \ \Longrightarrow\
\text{compatible families of coefficients}
   \ \Longrightarrow\
\text{homomorphisms between standard modules}.
\]
Thus the two papers address the two complementary ingredients described above by closely related elementary methods. Together they provide a uniform combinatorial machinery, based mainly on graph theory and linear algebra, for establishing Dold--Kan--type equivalences for generalized Reedy categories.

When the standard modules are projective and explicit normalization idempotents are available, their Hom-spaces can often be computed directly by Yoneda's lemma. The fiber-graph machinery of \cite{LiHom} becomes particularly useful when standard modules fail to be projective, or when explicit normalization idempotents are unavailable, since it computes Hom-spaces without requiring such a splitting. Such nonprojective phenomena already occur naturally in several Dold--Kan--type settings; see, for instance, the cyclic theory of Dwyer and Kan \cite{DwyerKan1,DwyerKan2}, the representation theory of finite sets and all maps studied by Powell and Vespa \cite{Powell, PowellVespa}, and the examples considered in \cite{LiRel}. Thus the projective-normalization framework of the present paper and the more general fiber-graph machinery of \cite{LiHom} should be viewed as complementary parts of the same representation-theoretic approach.

\subsection{Main results}

We now briefly describe the main results of the paper. Let $\C$ be a locally finite category with a generalized Reedy structure $(\Cp, \Cm, d)$, $k$ a commutative ring, and
\[
P_x=k\C(x,-), \qquad \Delta_x=P_x/\I_x
\]
be the representable and standard modules associated with an object $x$, where $\I_x$ is the submodule generated by all noninvertible morphisms in $\Cm$ with source $x$. The central construction of the paper is a finite bipartite graph $\mathcal N_x$, called the \emph{normalization graph}. Its black vertices are the endomorphisms of $x$, while its white vertices record the fibers of left composition by unfactorizable morphisms in $\Cm$ with source $x$. More precisely, if $s:x\to y$ is such a morphism and $h$ lies in the image of
\[
s^{\ast}:\End_{\C}(x)\longrightarrow \C(x,y),\qquad f\longmapsto sf,
\]
then the corresponding white vertex records the fiber
\[
F_{s,h}=\{f\in\End_{\C}(x)\mid sf=h\}.
\]
The automorphisms of $x$ form the boundary of $\mathcal N_x$.

A weighting on the black vertices is called \emph{balanced} if the sum of its coefficients over every normalization fiber is zero. Our first main result identifies these balanced weightings with homomorphisms from the standard module to the representable module (Proposition \ref{prop:balanced-hom}):
\[
\Bal_k(\mathcal N_x) \cong \Hom_{k\C}(\Delta_x,P_x).
\]
Under this correspondence, projectivity of $\Delta_x$ becomes a boundary-value problem on the normalization graph: $\Delta_x$ is projective if and only if there exists a balanced weighting whose restriction to the boundary is the delta function at the identity (we call it a \emph{normalized weighting}). Equivalently, the boundary map
\[
\partial_x:\Bal_k(\mathcal N_x)\longrightarrow kG_x
\]
is surjective; see Theorem \ref{thm:main}. Thus the projectivity problem for standard modules is reduced to solving an explicit finite system of linear equations.

The second main result shows that a solution of this linear problem automatically produces a genuine normalization idempotent. Explicitly, given a normalized weighting $\epsilon$ on $\mathcal{N}_x$, set
\[
\Phi_x = \sum_{f \in \End_{\C}(x)} \epsilon(f)f.
\]
Then
\[
\Phi_x^2 = \Phi_x, \qquad P_x \Phi_x \cong \Delta_x, \qquad \I_x = P_x (1 - \Phi_x).
\]
Consequently, the classical normalization operators can be viewed as explicit solutions of the linear boundary-value problem encoded by $\mathcal N_x$. For details, see Theorem \ref{thm:normalization-correspondence}.

We also characterize the uniqueness of normalization idempotents. Once a normalization exists at $x$, it is unique if and only if the singular ideal $kD_x$ has a two-sided identity. Assuming that every object admits a normalization, we further show that the standard modules induce an equivalence
\[
\C \Mod \simeq \prod_{x \in \Ob(\C)} kG_x \Mod
\]
if and only if every positive morphism admits a left inverse and the normalization idempotent at every object is unique. This gives a necessary and sufficient criterion, within our generalized Reedy framework, for the standard-module construction to yield an equivalence with representations of the core groupoid, complementing Street's sufficient criterion
in \cite{St}. For details, see Corollary~\ref{cor:normalization-uniqueness} and Theorem~\ref{thm:core-groupoid-criterion}.

We illustrate this machinery by several examples.
\begin{itemize}
\item For the simplex category, the normalization equations reduce to the alternating weighting on a Boolean cube, recovering the classical simplicial normalization projector in \cite{Dold1958, Kan1958}.

\item For the span categories considered in \cite{LiHom}, the commuting diagonal idempotents arising from proper subobjects give an explicit normalized weighting, showing that all standard modules are projective over an arbitrary commutative coefficient ring.

\item For the category $\VA_q$ of finite-dimensional $\mathbb F_q$-vector spaces and linear maps, assuming that $q$ is invertible in $k$, we show that Kov\'{a}cs' construction in \cite{Kovacs} of the identity in the singular matrix ideal may be interpreted as a solution of the linear equations associated with the normalization graph, and the complementary Kov\'{a}cs–Kuhn idempotent is precisely the associated normalization idempotent.

\item Finally, we apply a slight variant of the criterion to Pirashvili's projective generators for $\Gamma$-modules in \cite{Pirashvili}. The relevant normalization graph again contains a Boolean cube, whose alternating weighting gives the usual normalization operator. This recovers the projective generators $(t^*)^{\otimes n}$, and hence Pirashvili's Dold--Kan--type equivalence.
\end{itemize}

As a further application, we consider the category $\AF_q$ of finite-dimensional affine spaces over $\mathbb F_q$. The linear Kov\'{a}cs--Kuhn normalization extends to the affine endomorphism monoid by assigning weight zero to all endomorphisms with nonzero translation part, and hence again yields normalized projective standard modules. However, unlike the linear case, these standard modules are not Hom-orthogonal. More precisely, if $\Delta_n$ denotes the standard module attached to the $n$-dimensional affine space, then $\Hom_{k\AF_q} (\Delta_n,\Delta_m)$ vanishes unless $m = n$ or $m = n-1$. In the latter case, the Hom-space as a $k$-module is freely generated by the affine hyperplane embeddings whose images do not contain 0. Consequently, the representation category of $\AF_q$ is equivalent to the representation category of a thin $k$-linear category with nonzero morphisms only in degrees $0$ and $-1$, and every composite of two degree $-1$ morphisms is zero; see Theorem \ref{thm:affine-dold-kan}.

The affine example also leads to a useful transfer principle for normalization idempotents. If $\D$ is a wide generalized Reedy subcategory of $\C$ and every negative morphism of $\C$ is obtained from one in $\D$ by postcomposition with an automorphism of its codomain, then any family of normalization idempotents for $\D$ also serves as a family of normalization idempotents for $\C$. Besides providing a conceptual explanation for the passage from linear to affine spaces, this principle applies to semilinear and semiaffine enlargements. Combined with the computation of Hom-spaces, this yields Dold--Kan--type equivalences for the categories of semilinear and semiaffine maps. Together with the reconstruction result established in \cite{LiRel}, these equivalences provide transparent descriptions of uniformly continuous representations of the corresponding infinite transformation monoids. For details, see Proposition~\ref{prop:transfer-normalization}, Example~\ref{examples}, and Corollaries~\ref{cor:semilinear-semiaffine} and~\ref{cor:universal-transformation-monoids}.

\subsection{Organization of the paper}

The paper is organized as follows. In Section~2, we recall the necessary background on generalized Reedy categories and standard modules, introduce normalization graph and its normalized weightings, and establish the main criteria for projectivity of standard modules as well as the construction of normalization idempotents. We apply the general machinery to several examples in Section 3, and describe new applications to several categories in finite geometry in the last section.

\subsection{Acknowledgement on the use of AI}

The research program, principal ideas and strategy, and the main mathematical constructions of this work were conceived by the author. During the preparation of the paper, ChatGPT (OpenAI) was used as an auxiliary tool for exploratory mathematical discussion, assistance in locating relevant literature, checking intermediate arguments, and improving the exposition. All mathematical statements and proofs, as well as all bibliographic references, were independently written and verified by the author, who takes full responsibility for the content of the paper.

\section{The main construction and results}

\subsection{Generalized Reedy categories and their representations}

In this section we describe some elementary definitions and facts on generalized Reedy categories and their representation theory. Recall from \cite{BergerMoerdijk2011} that a \emph{generalized Reedy structure} on a small skeletal category $\C$ is a triple $(\Cp, \Cm, d)$ with $d: \Ob(\C) \to \mathbb N$ a \emph{degree map}, and $\Cp$ and $\Cm$ wide subcategories satisfying the following axioms:
\begin{enumerate}[label=(R\arabic*)]
\item Every morphism $f$ of $\C$ admits a factorization $f = f^+ f^-$ with $f^+$ a morphism in $\Cp$ and $f^-$ a morphism in $\Cm$, and this factorization is unique up to an intermediate automorphism.

\item Every noninvertible morphism in $\Cp$ strictly raises degree, and every noninvertible morphism in $\Cm$ strictly lowers degree.

\item $\Cp \cap \Cm$ is the subgroupoid of isomorphisms.

\item If $f \in \Cm(x,y)$ and $\sigma \in \Aut_{\C}(y)$ satisfy $\sigma f = f$, then $\sigma = 1_y$.
\end{enumerate}
To suppress technical issues, we also assume that $\C$ is locally finite, so every Hom-set is finite, and that for each $n \in \mathbb{N}$, there are only finitely many objects $x$ such that $d(x) = n$.

For an object $x$, write
\[
G_x=\Aut_{\C}(x),\qquad E_x=\End_{\C}(x) = \C(x,x), \qquad D_x = E_x \setminus G_x.
\]
It is clear from the above definition that
\[
\C(x, y) \cong \bigsqcup_{z} \Cp(z, y) \times_{G_z} \Cm(x, z) = \bigsqcup_{\substack{d(z) \leqslant d(x) \\ d(z) \leqslant d(y) }} \Cp(z, y) \times_{G_z} \Cm(x, z).
\]

\begin{definition}
A noninvertible morphism $f: x \to y$ in $\Cm$ is called \emph{unfactorizable} if it cannot be written as $f = gh$ with both $g$ and $h$ noninvertible morphisms in $\Cm$. Similarly, one can define unfactorizable morphisms in $\Cp$.
\end{definition}

The following elementary properties of unfactorizable morphisms have been established in \cite{LiHom}.
\begin{enumerate}
\item every noninvertible morphism in $\Cp$ or $\Cm$ can be written as a finite composite of unfactorizable morphisms.

\item The composite of an unfactorizable morphism and an automorphism is still unfactorizable.
\end{enumerate}
Consequently, for $x, y \in \Ob(\C)$, unfactorizable morphisms in $\Cm(x, y)$ form a left $G_y$-set, so we choose one representative from each orbit to get a set $\mathbb{U}_{x, y} = \{h_i\}_{i \in I}$ of representative unfactorizable morphisms in $\Cm(x, y)$, and define
\[
\mathbb{U}_x = \bigsqcup_{y \in \Ob(\C)} \mathbb{U}_{x, y},
\]
which is a finite set since there are only finitely many objects $x$ for which $\Cm(x, y) \neq \varnothing$ since $d(y) \leqslant d(x)$.

Now we turn to representation theory of $\C$. Let $k$ be a commutative ring with identity. By definition, a \emph{$\C$-module} (over $k$) is a covariant functor from $\C$ to $k \Mod$, the category of $k$-modules. The \emph{representable module} at an object $x$ is $ P_x = k\C(x,-)$, which is projective in $\C \Mod$. Let $\I_x \subseteq P_x$ be the submodule generated by all noninvertible morphisms in $\Cm$ with source $x$. Explicitly, $\I_x(y)$ is the $k$-submodule of $k\C(x,y)$ spanned by the morphisms $f:x\to y$ such that $f^-$ in the Reedy factorization of $f$ is noninvertible. The \emph{standard module} attached to $x$ is $\Delta_x = P_x/\I_x$. Denote the natural quotient map by $q_x: P_x \to \Delta_x$. By the definitions, one clearly has $\I_x(x) = kD_x$ and $\Delta_x(x) \cong kG_x$ as $k$-modules.

\subsection{Normalization graphs and normalized weightings}

In this subsection we introduce the main combinatorial construction used to characterize projectivity of standard modules.

Given an object $x$ and a representative unfactorizable morphism $s: x \to y$ in $\mathbb{U}_x$, left composition with $s$ defines a map of finite sets
\[
s^{\ast}: E_x \longrightarrow \C(x,y),\qquad f \longmapsto sf.
\]
For each $h \in \im(s^{\ast})$, define the fiber
\[
F_{s,h} = \{ f \in E_x \mid sf = h\}.
\]

\begin{definition} \label{def:normalization-graph}
The \emph{normalization graph} $\mathcal N_x$ is the bipartite graph with the following vertices and edges:
\begin{itemize}
\item Its black vertices are endomorphisms of $x$, namely $V_{\mathrm b}(\mathcal N_x) = E_x$.

\item Its white vertices are nonempty elementary fibers:
\[
V_{\mathrm w}(\mathcal N_x) = \{F_{s,h} \mid s \in \mathbb{U}_x,\ F_{s, h} \neq \varnothing\}.
\]

\item A black vertex $f$ is joined to a white vertex $F_{s,h}$ precisely when $f \in F_{s,h}$.
\end{itemize}
The subset $\partial\mathcal N_x = G_x$ of black vertices is called the \emph{boundary}; the vertices in $D_x = E_x \setminus G_x$ are called \emph{interior vertices}.
\end{definition}

Thus every white vertex simply records one fiber of one unfactorizable morphism. If two unfactorizable maps differ only by postcomposition with an automorphism of their codomain, the corresponding fibers are essentially the same. Thus we may use only elements in $\mathbb{U}_x$ to construct the graph.

\begin{definition} \label{def:balanced}
A \emph{balanced $k$-weighting} of $\mathcal N_x$ is a function $\epsilon: E_x \to k$ such that
\[
\sum_{f \in F_{s,h}} \epsilon(f) = 0
\]
for every white vertex $F_{s,h}$. Denote the $k$-module of balanced weightings by $\Bal_k(\mathcal N_x)$.
\end{definition}

Given a weighting $\epsilon:E_x\to k$, we obtain an element
\[
\Phi_{\epsilon} = \sum_{f \in E_x} \epsilon(f)f
\]
in $kE_x$. Conversely, given an element $\Phi \in kE_x$, its coefficient function is a weighting. Thus weightings correspond bijectively to elements in $kE_x$.

The following result gives a close relationship between balanced $k$-weightings and normalization operators.

\begin{lemma}\label{lem:balanced-annihilator}
The following are equivalent:
\begin{enumerate}
\item $\epsilon$ is balanced;

\item $s\Phi_{\epsilon}=0$ for every $s\in\mathbb U_x$;

\item $u\Phi_{\epsilon}=0$ for every noninvertible morphism $u$ with source $x$ in $\Cm$.
\end{enumerate}
\end{lemma}

\begin{proof}
Fix $s\in\mathbb U_x$. Grouping the terms according to the value of the composite $sf$, we obtain
\[
s\Phi_{\epsilon} = \sum_{f\in E_x}\epsilon(f)\,sf = \sum_{h\in\im(s^\ast)} \left( \sum_{f\in F_{s,h}}\epsilon(f) \right)h.
\]
Therefore $s\Phi_{\epsilon}=0$ if and only if the coefficient of every morphism $h\in\im(s^\ast)$ vanishes, which is precisely the balancedness condition. Hence {\rm (i)} and {\rm (ii)} are equivalent. The equivalence of {\rm (ii)} and {\rm (iii)} follows from the fact that every noninvertible morphism in $\Cm$ can be written as a finite composite of unfactorizable morphisms, and each unfactorizable morphism $s: x \to z$ can further be written as a composite $\sigma s'$ with $\sigma \in G_z$ and $s' \in \mathbb{U}_x$.
\end{proof}

Next we show that balanced weightings on normalization graphs are exactly the morphisms from the standard module to the representable. Given $\Phi \in kE_x$, let $R_\Phi: P_x \to P_x$ be the unique natural transformation given by Yoneda's lemma.

\begin{proposition}\label{prop:balanced-hom}
There is a canonical $k$-module isomorphism
\[
 \Bal_k(\mathcal N_x)  \cong \Hom_{k\C}(\Delta_x, P_x).
\]
Under this isomorphism, a balanced weighting $\epsilon$ corresponds to the unique map $j_{\epsilon}: \Delta_x \to P_x$ satisfying $j_{\epsilon} q_x = R_{\Phi_{\epsilon}}$.
\end{proposition}

\begin{proof}
A natural transformation $R_\Phi:P_x\to P_x$ factors through the quotient $q_x:P_x\to\Delta_x$ if and only if it vanishes on $\I_x$. Since $\I_x$ is generated by the unfactorizable morphisms $u$ in $\Cm$, this is equivalent to $u \Phi = 0$ for every such $u$. By Lemma \ref{lem:balanced-annihilator}, this is further equivalent to the coefficient function of $\Phi$ being balanced. Since $q_x$ is an epimorphism, the factorization through $\Delta_x$ is unique.
\end{proof}

Recall that boundary vertices in $\mathcal{N}_x$ are precisely elements in $G_x$. Thus we define a boundary map
\[
\partial_x: \Bal_k(\mathcal N_x) \longrightarrow kG_x, \quad \partial_x(\epsilon) = \sum_{g \in G_x} \epsilon(g)g.
\]
Equivalently, if $\epsilon$ corresponds to $\Phi_{\epsilon} \in P_x(x) = kE_x$, then $\partial_x(\epsilon) = q_x(\Phi_{\epsilon}) \in kG_x \cong \Delta_x(x)$.

\begin{definition}
A balanced weighting $\epsilon \in \Bal_k(\mathcal N_x)$ is called \emph{normalized} if its boundary value is the delta function at the identity, that is, $\epsilon(1_x) = 1$ and $\epsilon(g) = 0$ for every $g \in G_x \setminus \{1_x\}$.
\end{definition}

\subsection{Criteria for projectivity of standard modules}

Now we are ready to describe a combinatorial criterion for the projectivity of standard modules.

\begin{theorem}\label{thm:main}
For an object $x$ of $\C$, the following conditions are equivalent:
\begin{enumerate}[label=(\arabic*)]
\item The standard module $\Delta_x$ is projective.

\item There exists a normalized weighting $\epsilon \in \Bal_k(\mathcal N_x)$.

\item The boundary map $\partial_x$ is surjective.
\end{enumerate}
\end{theorem}

\begin{proof}
\noindent
$(1) \Rightarrow (2)$. Let $j: \Delta_x \to P_x$ be a section of $q_x$, which exists by the assumption. By Proposition~\ref{prop:balanced-hom}, the composite $jq_x: P_x \to P_x$ corresponds to a unique balanced weighting $\epsilon$, or equivalently to a unique element $\Phi_{\epsilon} \in kE_x$. Evaluating the identity $q_xj = \id_{\Delta_x}$ at the canonical generator $[1_x] \in \Delta_x(x)$ gives
\[
 q_x(\Phi_{\epsilon}) = [1_x].
\]
Under $\Delta_x(x) \cong kG_x$, this means exactly $\epsilon (1_x) = 1$ and $\epsilon (g) = 0$ for all $g \neq 1_x$, namely $\epsilon$ is normalized.

\medskip
\noindent
$(2) \Rightarrow (1)$. By Proposition~\ref{prop:balanced-hom}, one gets a map $j_{\epsilon}: \Delta_x \to P_x$ such that $j_{\epsilon}q_x = R_{\Phi_{\epsilon}}$. Since $[1_x] = q_x(1_x)$, evaluating this identity at $1_x$ gives
\[
j_{\epsilon}([1_x]) = j_{\epsilon}q_x(1_x) = R_{\Phi_{\epsilon}}(1_x) = \Phi_{\epsilon}.
\]
The normalized boundary condition says $q_x(\Phi_{\epsilon})=[1_x]$, so one has
\[
(q_xj_{\epsilon})([1_x]) = q_x(\Phi_{\epsilon}) = [1_x].
\]
Since $\Delta_x$ is generated by $[1_x]$, it follows that $q_xj_{\epsilon}=\id_{\Delta_x}$. Thus $q_x$ splits, and $\Delta_x$ is projective.

\medskip
\noindent
$(2) \Leftrightarrow (3)$. Clearly (3) implies (2). It remains to prove (2) implies (3). Suppose that $\epsilon$ is normalized. For $g \in G_x$, right multiplication by $g$ preserves balanced elements: if $\Phi_{\epsilon}$ is annihilated by every unfactorizable $s$, then
\[
 s(\Phi_{\epsilon} g) = (s\Phi_{\epsilon})g = 0.
\]
Moreover the boundary value of $\Phi_{\epsilon} g$ is $g$. Hence every basis element $g \in G_x$ lies in the image of $\partial_x$, so $\partial_x$ is surjective.
\end{proof}

We can use the following elementary criterion to show that standard modules are not projective for many concrete examples.

\begin{corollary}\label{cor:identity-fiber-obstruction}
Let $x$ be an object of $\C$. If there exists an unfactorizable morphism $s: x \to y$ in $\Cm$ such that the normalization fiber
\[
F_{s,s} = \{f \in E_x \mid sf=s \}
\]
is contained in $G_x$, then $\Delta_x$ is not projective.
\end{corollary}

\begin{proof}
By Theorem \ref{thm:main}, projectivity of $\Delta_x$ gives a normalized weighting $\epsilon: E_x \to k$. Since $1_x \in F_{s,s}$, balancedness gives
\[
0 = \sum_{f\in F_{s,s}} \epsilon(f).
\]
The total contribution from automorphisms in this fiber is equal to $1$ since $\epsilon$ is normalized. Hence
\[
\sum_{f\in F_{s,s}\cap D_x} \epsilon(f) = -1.
\]
Therefore $F_{s,s} \cap D_x \neq \varnothing$, proving the claim.
\end{proof}

Theorem \ref{thm:main} has a parallel linear algebraic version. Before describing it, let us introduce some notation. Let $W_x$ be the set of white vertices in $\mathcal N_x$, namely the set of normalization fibers. Define the incidence matrix
\[
B_x = (b_{F, f})_{F \in W_x,\ f \in E_x}
\]
by
\[
 b_{F,f}=
 \begin{cases}
 1,& f \in F,\\
 0,& f \notin F.
 \end{cases}
\]
Equivalently, regard $B_x$ as a $k$-linear map
\[
 B_x: k^{E_x} \longrightarrow k^{W_x}, \qquad (B_x \epsilon)(F) = \sum_{f \in F} \epsilon(f).
\]
Then $\Bal_k(\mathcal N_x) = \ker B_x$. Separate the columns according to $E_x = G_x \sqcup D_x$ and write
\[
B_x =
\begin{bmatrix}
B_x^G & B_x^D
\end{bmatrix}.
\]
For $g \in G_x$, let $b_g$ denote the column of $B_x$ indexed by $g$.

\begin{corollary} \label{thm:matrix}
The following are equivalent:
\begin{enumerate}
\item $\Delta_x$ is projective;

\item $b_{1_x} \in \im(B_x^D)$;

\item $\im(B_x^G) \subseteq \im(B_x^D)$.
\end{enumerate}
If $k$ is a field, these are further equivalent to $\operatorname{rank} B_x = \operatorname{rank}B_x^D$.
\end{corollary}

\begin{proof}
A weighting satisfying the prescribed boundary condition has the form
\[
 \epsilon = \delta_{1_x} + \rho, \qquad \rho \in k^{D_x},
\]
where $\delta_{1_x}$ is the delta function at the identity automorphism. It is balanced precisely when
\[
 B_x \epsilon = b_{1_x} + B_x^D d = 0.
\]
Thus such a weighting exists if and only if $b_{1_x} \in \im(B_x^D)$. The equivalence of (1) and (2) then follows by Theorem~\ref{thm:main}.

Clearly, (3) implies (2). Conversely, suppose (2) holds. By Theorem \ref{thm:main}, there is a $\Phi_x \in kE_x$ whose coefficient function is a normalized weighting. For every $g\in G_x$, the corresponding coefficient function of $\Phi_xg$ is balanced and has boundary value $g$. Translating this statement to incidence matrices shows $b_g \in \im(B_x^D)$ for every $g\in G_x$. Hence (3) holds.

If $k$ is a field, condition (3) says exactly that adjoining the automorphism columns to $B_x^D$ does not increase the column rank. Thus it is equivalent to the equality of ranks.
\end{proof}

The following general fact plays a vital role for us to obtain Dold-Kan-type equivalences.

\begin{lemma} \label{lem:projective generators}
Suppose that every standard module $\Delta_x$ is projective. Then the family $\{\Delta_x \mid x \in \Ob(\C)\}$ is a family of small projective generators of $\C \Mod$. More precisely, for every object $x$, the representable module $P_x$ belongs to the additive closure of the standard modules $\Delta_y$ with $d(y) \leqslant d(x)$.
\end{lemma}

\begin{proof}
We argue by induction on $d(x)$. If $d(x)=0$, then there is no noninvertible morphism in $\Cm$ with source $x$, so $\I_x = 0$ and hence $P_x = \Delta_x$.Suppose that $d(x)>0$ and that the assertion holds for all objects of smaller degree. Since $\Delta_x$ is projective, the canonical quotient map $P_x \twoheadrightarrow \Delta_x$ splits, so $P_x \cong \I_x \oplus \Delta_x$, and in particular $\I_x$ is projective.

Recall that $\I_x$ is generated by the noninvertible morphisms in $\Cm$ with source $x$. Since every such morphism is a composite of unfactorizable morphisms, it is already generated by the representative unfactorizable morphisms $s: x \to y$ in $\mathcal U_x$. Consequently, composition with these morphisms induces an epimorphism
\[
\bigoplus_{s\in\mathcal U_x}P_{t(s)} \longrightarrow \I_x, \qquad g \longmapsto gs,
\]
where $t(s)$ denotes the codomain of $s$. The sum is finite by our local finiteness assumptions, and $d(t(s)) < d(x)$ for every $s\in\mathcal U_x$.

Since $\I_x$ is projective, the above epimorphism splits. Hence $\I_x$ is a direct summand of
\[
\bigoplus_{s \in\mathcal U_x}P_{t(s)}.
\]
By the induction hypothesis, every $P_{t(s)}$ belongs to the additive closure of the standard modules $\Delta_y$ with $d(y) < d(x)$. Therefore the same is true of $\I_x$. Since $P_x \cong \I_x \oplus\Delta_x$, it follows that $P_x$ belongs to the additive closure of the standard modules $\Delta_y$ with $d(y) \leqslant d(x)$. Thus standard modules form a family of projective generators. Moreover, every $\Delta_x$ is small since it is isomorphic to a direct summand of the small projective representable $P_x$.
\end{proof}

\begin{definition}
Let $x$ be an object of $\C$. A \emph{normalization idempotent} for $\Delta_x$ is an element $\Phi \in kE_x$ such that $1_x - \Phi \in kD_x$ and $u \Phi = 0$ for every noninvertible morphism $u$ in $\Cm$ with source $x$.
\end{definition}

The terminology is justified by the following result, which provides a correspondence between normalization idempotents and normalized weightings.

\begin{theorem} \label{thm:normalization-correspondence}
Let $x$ be an object of $\C$. There is a natural bijection between the following sets:
\begin{enumerate}
\item normalized $k$-weightings $\epsilon: E_x \to k$,

\item normalization idempotents $\Phi \in kE_x$.
\end{enumerate}

More explicitly, the correspondence from (1) to (2) is
\[
\epsilon \longmapsto \Phi_{\epsilon} = \sum_{f \in E_x} \epsilon(f)f.
\]
Moreover, under this correspondence, one has
\[
\Phi_{\epsilon}^2 = \Phi_{\epsilon}, \quad P_x \Phi_{\epsilon} \cong \Delta_x, \quad \I_x = P_x(1_x - \Phi_{\epsilon}).
\]
\end{theorem}

\begin{proof}
Let $\epsilon: E_x \to k$ be a normalized weighting. By Lemma \ref{lem:balanced-annihilator}, one has $u \Phi_{\epsilon} = 0$ for every noninvertible morphism $u$ in $\Cm$ with source $x$. Moreover, under the quotient map $q_x: P_x \to \Delta_x$, the image of $\Phi_{\epsilon}$ in $\Delta_x(x) \cong kG_x$ is precisely $[1_x]$. This is equivalent to $1_x - \Phi_{\epsilon} \in kD_x = \I_x(x)$. Thus $\Phi_{\epsilon}$ is a normalization idempotent.

Conversely, given a normalization idempotent
\[
\Phi = \sum_{f \in E_x} c_f f \in kE_x,
\]
it defines a $k$-weighting
\[
c: E_x \longrightarrow k, \quad f \longmapsto c_f.
\]
Since $u \Phi = 0$ for every unfactorizable morphism $u$ in $\Cm$ with source $x$, Lemma \ref{lem:balanced-annihilator} implies that $c$ is balanced. Moreover, since $1_x - \Phi$ lies in $\I_x(x) = kD_x$, it follows that $c$ has boundary value $\delta_{1_x}$, and hence is a normalized weighting. Consequently, (1) and (2) are in bijective correspondence.

We next show the last statement. Let $\Phi$ be such an element. Since $1_x - \Phi$ is contained in $\I_x(x)$, every basis element $f$ occurring in $1_x - \Phi$ is a noninvertible morphism with source $x$, and hence factors as
\[
x \xrightarrow{u} z \xrightarrow{v} x,
\]
where $u$ is a noninvertible morphism in $\Cm$. By Lemma \ref{lem:balanced-annihilator},
\[
f\Phi = (vu)\Phi = v(u\Phi) = 0.
\]
By linearity, we deduce that $(1_x - \Phi) \Phi = 0$, so $\Phi^2 = \Phi$ as desired. Furthermore, the fact that $u \Phi = 0$ for every generator $u$ of $\I_x$ also implies that
\[
\I_x \subseteq P_x(1_x - \Phi).
\]
On the other hand, since $(1_x - \Phi) \in \I_x(x)$, we also have
\[
P_x (1_x - \Phi) \subseteq \I_x.
\]
Consequently, we have
\[
P_x(1_x - \Phi) = \I_x, \quad \Delta_x = P_x /\I_x \cong P_x \Phi
\]
as desired.
\end{proof}

\begin{remark}
Theorem \ref{thm:normalization-correspondence} shows that the normalization problem may be formulated as the linear boundary-value problem on $\mathcal N_x$: one seeks a balanced weighting $\epsilon: E_x \to k$ whose restriction to the boundary $G_x$ is $\delta_{1_x}$. Once such a weighting is found, the associated element
\[
\Phi_{\epsilon} = \sum_{f \in E_x} \epsilon(f) f
\]
is automatically a normalization idempotent. The normalization graph therefore replaces the problem of finding a suitable idempotent by the more intrinsic problem of solving a system of linear fiber-sum equations.
\end{remark}

Once a normalization idempotent exists, its uniqueness is equivalent to the singular ideal having a two-sided identity.

\begin{corollary}\label{cor:normalization-uniqueness}
Let $x$ be an object admitting a normalization idempotent $\Phi_x$. Then the set of normalization idempotents at $x$ is precisely
\[
\Phi_x + \Phi_x (kD_x).
\]
Consequently, the following conditions are equivalent:
\begin{enumerate}
\item The normalization idempotent at $x$ is unique.

\item $\Phi_x (kD_x) = 0$.

\item $e_x = 1_x - \Phi_x$ is a two-sided identity of $kD_x$.
\end{enumerate}
When these conditions hold, $e_x$ and $\Phi_x$ are central in $kE_x$, and we have an algebra decomposition
\[
kE_x \cong kD_x \times kG_x.
\]
\end{corollary}

\begin{proof}
Since $\Phi_x$ is a normalization idempotent, we have $1_x-\Phi_x \in kD_x$ and $kD_x \Phi_x = 0$. For every $a\in J_x$, the element $\Phi_x+\Phi_xa$ has the same coefficients on $G_x$ as $\Phi_x$, and satisfies the
normalization equations, since
\[
s(\Phi_x+\Phi_xa)=s\Phi_x+(s\Phi_x)a=0
\]
for every noninvertible negative morphism $s$ with source $x$. Thus it is a normalization idempotent.

Conversely, let $\Psi_x$ be another normalization idempotent. Then
\[
\Psi_x = \Phi_x + (\Psi_x - \Phi_x) = \Phi_x + \Phi_x (\Psi_x - \Phi_x) + (1_x - \Phi_x) (\Psi_x - \Phi_x).
\]
Since $1_x - \Phi_x \in kD_x$, and $kD_x \Psi_x = 0 = kD_x \Phi_x$, it follows that
\[
\Psi_x = \Phi_x + \Phi_x (\Psi_x - \Phi_x) \in \Phi_x + \Phi_x (kD_x)
\]
proving the first statement.

The equivalence of (1) and (2) follows from the above description of normalization idempotents. Since $J_x\Phi_x=0$, the element $e_x=1_x-\Phi_x$ is already a right identity of $J_x$, and it is a left identity precisely when $\Phi_xJ_x = 0$. Thus (2) and (3) are equivalent. The last statement follows immediately.
\end{proof}

The parallel result in terms of linear algebra is:

\begin{corollary}\label{cor:matrix-uniqueness}
The normalization graph $\mathcal N_x$ admits a unique normalized weighting if and only if $b_{1_x} \in \im(B_x^D)$ and $B_x^D$ is injective. In the case that $k$ is a field, these conditions are equivalent to
\[
\operatorname{rank}B_x = \operatorname{rank}B_x^D = |D_x|.
\]
\end{corollary}

\begin{proof}
A normalized weighting has the form $\epsilon = (\delta_{1_x}, u)$, where $u \in k^{D_x}$ satisfies
\[
B_x^D u=-b_{1_x}.
\]
If $u_0$ is one solution, the complete solution set is $u_0 + \ker B_x^D$. The assertions follow from Corollary \ref{thm:matrix} and, over a field, the rank--nullity theorem.
\end{proof}

Under a left-inverse hypothesis on positive morphisms, uniqueness of normalization idempotents also characterizes when the standard-module construction reduces the representation category to that of the core groupoid.

\begin{theorem}\label{thm:core-groupoid-criterion}
Suppose that every object of $\C$ admits a normalization idempotent. Then the following conditions are equivalent:
\begin{enumerate}
\item Every morphism in $\Cp$ admits a left inverse in $\C$, and every object admits a unique normalization idempotent.

\item Every morphism in $\Cp$ admits a left inverse in $\C$, and every singular ideal $kD_x$ has a two-sided identity.

\item The standard modules are pairwise Hom-orthogonal, namely $\Hom_{k\C}(\Delta_x, \Delta_y) = 0$ when $x \not\cong y$.

\item The functor
\[
\C \Mod \longrightarrow \prod_{x \in \Ob(\C)} kG_x \Mod, \qquad M \longmapsto \bigl(\Hom_{k\C}(\Delta_x, M)\bigr)_{x}
\]
is an equivalence of categories.
\end{enumerate}
\end{theorem}

\begin{proof}
$(1)\Leftrightarrow(2)$ follows from Corollary \ref{cor:normalization-uniqueness}.

\medskip
\noindent
$(1)\Rightarrow(3)$. Choose the unique normalization idempotent $\Phi_x$ at each object. By Corollary~\ref{cor:normalization-uniqueness},
\[
(kD_x) \Phi_x = 0 = \Phi_x(kD_x).
\]
Using $P_x\Phi_x\cong\Delta_x$ and Yoneda's lemma, we have
\[
\Hom_{k\C}(\Delta_x, \Delta_y) \cong \Hom_{k\C}(P_x \Phi_x, P_y \Phi_y) \cong \Phi_x P_y \Phi_y = \Phi_x\,k\C(y,x)\,\Phi_y.
\]
For $f\in\C(y,x)$ with $x\not\cong y$, choose a Reedy factorization
\[
y\xrightarrow{f^-}z\xrightarrow{f^+}x.
\]
If $f^-$ is noninvertible, then $f^-\Phi_y=0$ by normalization. Otherwise $f^+$ is noninvertible. Choose $g \in \C(x,z)$ with $gf^+ = 1_z$. Then $f^+g$ is noninvertible, and hence belongs to $D_x$, so we have
\[
\Phi_x f^+ = \Phi_x(f^+g)f^+ = 0.
\]
In either case $\Phi_x f \Phi_y = 0$, proving (3).

\medskip
\noindent
$(3)\Rightarrow(1)$.
By the proof of Lemma~\ref{lem:projective generators}, the module $\I_x$ is a quotient of a finite direct sum of $\Delta_y$ with $d(y)<d(x)$. Projectivity and Hom-orthogonality of the standard modules therefore give
\[
\Hom_{k\C}(\Delta_x,\I_x)=0.
\]
Note that any two sections of $q_x: P_x \to \Delta_x$ differ by a morphism $\Delta_x\to\I_x$. Hence the section, and consequently the normalization idempotent, is unique.

It remains to show that every morphism in $\Cp$ has a left inverse. Suppose that $u: y \to x$ is a morphism in $\Cp$ admitting no left inverse. Let $N \subseteq \Delta_y$ be the submodule generated by the nonzero class $\overline{u} \in \Delta_y(x)$. For every $rf\in \C(x,y)$, the composite $ru$ is noninvertible; otherwise $(ru)^{-1}r$ would be a left inverse of $u$. Thus every such composite vanishes in $\Delta_y(y)$, giving $N(y) = 0$.

Choose an arbitrary object $z$. If $z \not \cong y$, then Hom-orthogonality gives
\[
\Hom_{k\C}(\Delta_z,N) \hookrightarrow \Hom_{k\C}(\Delta_z,\Delta_y) = 0;
\]
otherwise, we have
\[
\Hom_{k\C}(\Delta_y, N) \hookrightarrow \Hom_{k\C}(P_y, N) \cong N(y)=0.
\]
In either case, we have
\[
\Hom_{k\C} (\Delta_z, N) = 0.
\]
Since the standard modules generate $\C \Mod$ by Lemma~\ref{lem:projective generators}, this forces $N = 0$, contradicting $\overline u\neq0$. Hence every morphism in $\Cp$ admits a left inverse.

\medskip
\noindent
$(3)\Leftrightarrow(4)$.
The standard modules form a family of small projective generators
by Lemma~\ref{lem:projective generators}, and
\[
\End_{k\C}(\Delta_x) \cong (kG_x)^{\mathrm{op}} \cong kG_x.
\]
Thus (3) implies (4) by Morita theory. Conversely, the equivalence in (4) identifies $\Delta_x$ with the regular module in the $x$-factor, since these objects represent the corresponding component functors. Hence (3) follows.
\end{proof}

\begin{remark}
The preceding theorem is closely related to Street's criterion for when the core groupoid suffices. Street calls a two-sided identity of the singular endomorphism ideal a Kov\'{a}cs idempotent and shows that, under suitable hypotheses, its complement yields an equivalence with representations of the core groupoid; see \cite{St}. When $(\Cm, \Cp)$ is a proper factorization system satisfying his finiteness hypothesis, condition~(1) of our theorem implies Street's Idempotent Axiom, so the implication from (1) to (4) can also be deduced from his result. Our formulation identifies the existence of these singular-ideal identities with uniqueness of normalized weightings and proves the converse: the standard-module equivalence forces both uniqueness of normalization idempotents and the existence of left inverses for positive morphisms.
\end{remark}

\begin{remark}
The left-inverse hypothesis cannot be omitted. For a direct category $\C$ (namely $\C = \Cp$), one has $kD_x = 0$ for every $x \in \Ob(\C)$ since $D_x = \varnothing$. Moreover, the unique normalization idempotent is $1_x$. Nevertheless, the standard modules need not be Hom-orthogonal. For an example, the reader can consider the semisimplex category.
\end{remark}

\section{Examples}\label{sec:examples}

We now apply the normalization-graph criterion to several categories appearing in the motivating examples. Unlike the classical approaches in the literature (see for instance \cite{Kovacs, Kuhn}), which begin with the normalization idempotent and uses it to split the representable functor, we reverse the order as follows:
\[
\text{normalization graph} \Longrightarrow \text{normalized weightings} \Longrightarrow \text{normalization idempotents}.
\]

Most results in this section have been established in the literature or can be deduced from existing theories, so we do not claim originality for them. Our purpose is not to reprove these results for their own sake, but to show that their normalization constructions arise uniformly from the combinatorial framework developed in this paper.

\subsection{The simplex category}\label{subsec:simplex}

Let $\C=\mathrm{OA}$ be the simplex category in the shifted convention $[n] = \{1,\ldots,n\}$. Since $G_n$ is the trivial group for each $n$, the boundary of the normalization graph $\mathcal N_n$ consists only of the identity endomorphism.

For $1 \leqslant i \leqslant n-1$, let $s_i: [n] \twoheadrightarrow [n-1]$ be the surjective morphism in $\Cm$ identifying $i$ and $i+1$. These are all unfactorizable morphisms in $\Cm$ with source $[n]$. Given an order-preserving map $h: [n] \to [n-1]$, the corresponding normalization fiber is
\[
F_{i,h} = \{f \in E_n \mid s_if = h\}.
\]
Since $h$ is order preserving, the fiber $h^{-1}(i)$ is either empty or an interval. Suppose that
\[
h^{-1}(i)=\{a,a+1,\ldots,b\}
\]
has cardinality $m=b-a+1$. Outside this interval, a lift $f:[n] \to [n]$ satisfying $s_if=h$ is uniquely determined by:
\[
f(j)=
\begin{cases}
h(j),&h(j)<i,\\
h(j)+1,&h(j)>i.
\end{cases}
\]
On $h^{-1}(i)$, the value of $f$ must lie in $\{i,i+1\}$. Since $f$ is order preserving, there is a unique integer $0 \leqslant r \leqslant m$ such that $f$ takes the value $i$ on the first $r$ elements of $h^{-1}(i)$ and
the value $i+1$ on the remaining ones. Thus
\[
F_{i,h} = \{f_{i,h}^{(0)},f_{i,h}^{(1)},\ldots,f_{i,h}^{(m)}\}.
\]
The normalization equation for a weighting $\epsilon: E_n \to k$ is:
\begin{equation}\label{eq:oa-full-fiber-equation}
\sum_{r=0}^{m} \epsilon\bigl(f_{i,h}^{(r)}\bigr) = 0.
\end{equation}
Together with the boundary condition $\epsilon(1_n) = 1$, these equations determine the normalization problem.

We now look for a solution supported on a particularly simple family of endomorphisms. For every subset $S\subseteq[n-1]$, define $e_S \in E_n$ by
\[
e_S(j)=
\begin{cases}
j+1,&j\in S,\\
j,&j\notin S.
\end{cases}
\]
Note that $e_ \varnothing = 1_n$ and $e_S$ is noninvertible when $S \neq \varnothing$. Set
\[
T_n=\{e_S\mid S\subseteq[n-1]\}.
\]
We seek a solution of the normalization equations completely supported on $T_n$.

For $S,T\subseteq[n-1]$, one has $s_ie_S = s_ie_T$ if and only if $S\setminus\{i\} = T \setminus\{i\}$. Indeed, changing the membership of $i$ changes only the value of $e_S(i)$ between $i$ and $i+1$, which is invisible after applying $s_i$; membership of every $j\neq i$ remains detectable. Consequently, every normalization fiber meets $T_n$ either trivially or in a pair $\{e_S, e_{S\cup\{i\}}\}$ when $i$ is not contained in $S$. After imposing the condition that $\epsilon$ vanishes outside $T_n$, the nontrivial normalization equations \eqref{eq:oa-full-fiber-equation} reduce to
\begin{equation}\label{eq:oa-boolean-equations}
\epsilon(e_S) + \epsilon(e_{S\cup\{i\}}) = 0, \qquad i \notin S.
\end{equation}
Thus the remaining system is the edge system of the Boolean $(n-1)$-cube. Starting from the boundary value $\epsilon(e_\varnothing) = 1$, equation \eqref{eq:oa-boolean-equations} recursively forces
\[
\epsilon(e_S) = (-1)^{|S|} \qquad (S\subseteq[n-1]).
\]
We have therefore obtained a weighting
\[
\epsilon(f)=
\begin{cases}
(-1)^{|S|}, & f =e_S \text{ for some } S \subseteq [n-1],\\
0, & f \notin T_n.
\end{cases}
\]

It is routine to check that this is a normalized weighting. By Theorem~\ref{thm:normalization-correspondence}, it determines a normalization idempotent
\begin{equation}\label{eq:oa-normalization-from-graph}
\Phi_n = \sum_{S \subseteq[n-1]}(-1)^{|S|}e_S,
\end{equation}
which is precisely the usual simplicial normalization operator. Thus the classical normalization operator is recovered from the alternating balanced weighting on the Boolean cube contained in the normalization graph.

\subsection{Span categories}\label{subsec:spans}

We next consider the span categories of~\cite{LiHom}. Let $\mathscr D$ be a skeletal EI category such that its object poset is artinian, all Hom-sets are finite, every pair of morphisms with common codomain admits a pullback, and every morphism of $\mathscr D$ is monic. Let $\C = \operatorname{Span}(\mathscr D)$ be the category whose morphisms $x\to y$ are isomorphism classes of spans $x \xleftarrow{a} z \xrightarrow{b} y$, with composition defined by pullback. Morphisms in $\Cm$ are represented by
\[
 s_a = [a,1_z]: x \longrightarrow z,
\]
and morphisms in $\Cp$ by spans of the form $[1_z,b]$. For details, see \cite{LiHom}.

For a monomorphism $a: z \to x$ in $\mathscr{D}$, define
\[
 r_a=[1_z,a]: z \longrightarrow x.
\]
Then $r_a$ is a right inverse of $s_a$. Set
\[
 e_a = r_as_a = [a,a] \in E_x,
\]
which is an idempotent in $\C$, and is the diagonal span determined by the subobject $a: z \to x$. Furthermore, given morphisms $a:z\to x$ and $b:w\to x$ in $\mathscr{D}$, the elements $e_a$ and $e_b$ commute. More precisely, if $P = z \times_x w$ is a pullback and $p: P \to x$ is the induced monomorphism, then $e_a e_b = e_b e_a = e_p$. Indeed, composition of the diagonal spans is computed by the pullback of the two maps into $x$. Let $P = z \times_x w$ and
\[
p = a\operatorname{pr}_z = b \operatorname{pr}_w: P \to x
\]
be the projection, then both composites $e_ae_b$ and $e_be_a$ are represented by the same diagonal span $x\xleftarrow{p}P\xrightarrow{p}x$. Hence they are equal to $e_p$.

The following result removes the field hypothesis from \cite[Theorem 4.5]{LiHom}.

\begin{proposition}\label{prop:span-projective}
For every object $x$ of $\C$, the standard module $\Delta_x$ is projective over every commutative coefficient ring $k$. Consequently, if $k\mathscr{D}(x, y)$ is a right projective $kG_x$-module for all $x, y \in \Ob(\C)$, then
\[
\C \Mod \, \simeq \, \prod_{x \in \Ob(\C)} kG_x \Mod.
\]
\end{proposition}

\begin{proof}
Fix an object $x$, choose representatives unfactorizable morphisms $s_a = [a,1]: x \to z$ with source $x$, and let $\mathcal A_x$ be the resulting finite set of monomorphisms $a: z \to x$ in $\mathscr{D}$. Define
\begin{equation}\label{eq:span-normalizer}
\Phi_x = \prod_{a \in\mathcal A_x}(1 - e_a).
\end{equation}
Since these factors commute, the order is irrelevant. For every $a\in\mathcal A_x$, since $s_ae_a=s_ar_as_a=s_a$, one has $s_a(1-e_a)=0$, and hence $s_a\Phi_x=0$ for every unfactorizable negative morphism $s_a$. By Lemma \ref{lem:balanced-annihilator}, the coefficient function of $\Phi_x$ is a balanced weighting on the normalization graph.

It remains to check the boundary condition. Expanding ~\eqref{eq:span-normalizer}, the empty product contributes $1_x$. Every nonempty product $e_{a_1} \cdots e_{a_r}$ is the diagonal idempotent associated with the iterated pullback of the proper subobjects $a_1,\ldots,a_r$. The resulting monomorphism into $x$ factors through $a_1$ and therefore cannot be an isomorphism. Hence every nonempty term is a noninvertible endomorphism of $x$. The coefficient function of $\Phi_x$ consequently has boundary value $\delta_{1_x}$, and hence is normalized.

Theorem~\ref{thm:main} now implies that $\Delta_x$ is projective. Moreover, Theorem~\ref{thm:normalization-correspondence} gives $P_x \Phi_x \cong \Delta_x$. Finally, by \cite[Theorem~4.5]{LiHom}, the standard modules are hom-orthogonal under the stated projectivity assumption on the $k\mathscr D(x,y)$, and the asserted Morita equivalence follows.
\end{proof}

\begin{example}
Let $\mathscr D$ be the category of finite sets and injections. Then $\operatorname{Span}(\mathscr D)$ is the category of finite sets and partial injections. For $x=[n]$, the elementary proper subobjects are
\[
[n] \setminus \{i\} \hookrightarrow [n], \qquad 1 \leqslant i \leqslant n.
\]
Let $e_S$ denote the partial identity with domain $S\subseteq[n]$. Then $e_Se_T = e_{S \cap T}$, and Proposition~\ref{prop:span-projective} gives
\[
\Phi_n = \prod_{i=1}^{n} (1-e_{[n]\setminus\{i\}}) = \sum_{T \subseteq[n]} (-1)^{|T|}e_{[n]\setminus T}.
\]
Thus the normalization operator is the M\"{o}bius alternating sum of the Boolean lattice of subsets of $[n]$.
\end{example}

\subsection{The category of finite cyclic groups of squarefree order}

Let $\C$ be the skeletal category whose objects are finite cyclic groups $C_n=\mathbb Z/n\mathbb Z$ with $n$ squarefree, and whose morphisms are group homomorphisms. We equip $\C$ with the epi--mono Reedy structure by taking $d(C_n)=n$ as degree.

For a squarefree integer $n$, let $\mathcal{P}_n$ be the set of prime divisors of $n$. Since
\[
C_n \cong \prod_{p \in \mathcal{P}_n} C_p,
\]
the Chinese remainder theorem gives
\[
E_n = \End(C_n) \cong \mathbb Z/n\mathbb Z \cong \prod_{p \in \mathcal{P}_n}\mathbb F_p
\]
as multiplicative monoids, and
\[
G_n = \Aut(C_n) \cong (\mathbb Z/n\mathbb Z)^\times \cong \prod_{p \in \mathcal{P}_n}\mathbb F_p^\times.
\]
For each $p \in \mathcal P(n)$, let $\pi_p: C_n \twoheadrightarrow C_{n/p}$ be the canonical quotient. Then the maps $\pi_p$ form a complete set of representative unfactorizable morphisms in $\Cm$ with source $C_n$.

We first describe the normalization graph $\mathcal N_n$. Under the above identifications, write an endomorphism of $C_n$ as
\[
a=(a_q)_{q \in \mathcal{P}_n}\in\prod_{q \in \mathcal{P}_n}\mathbb F_q.
\]
For a fixed $p \in \mathcal{P}_n$, composition with $\pi_p$ forgets precisely the $p$-coordinate. Thus, if
\[
b=(b_q)_{q \in \mathcal{P}_n,\ q\neq p} \in \prod_{\substack{q \in \mathcal{P}_n\\q\neq p}}\mathbb F_q,
\]
the corresponding normalization fiber is
\begin{equation}\label{eq:cyclic-normalization-fiber}
F_{p,b}
=
\left\{
(a_q)_{q \in \mathcal{P}_n}
\ \middle|
a_q=b_q\text{ for every }q\neq p
\right\}.
\end{equation}
In particular, $|F_{p,b}|=p$. Consequently, the black vertices of $\mathcal N_n$ form the finite grid
\[
\prod_{p \in \mathcal{P}_n}\mathbb F_p.
\]
Its boundary is
\[
\prod_{p \in \mathcal{P}_n}\mathbb F_p^\times,
\]
and its white vertices correspond to the coordinate lines obtained by fixing every coordinate except one.

We now solve the boundary-value problem. For every subset $S\subseteq\mathcal P(n)$, let $e_S\in E_n$ be the endomorphism whose $p$-coordinate is
\[
(e_S)_p=
\begin{cases}
0,&p\in S,\\
1,&p\notin S.
\end{cases}
\]
Then $e_\varnothing=1_{C_n}$ and $e_Se_T=e_{S\cup T}$. Define $T_n =\{e_S\mid S\subseteq\mathcal P(n)\}$, which is the Boolean cube inside the set of black vertices of the normalization graph.

By the same argument as in Subsection \ref{subsec:simplex}, the alternating weighting on this Boolean cube solves the boundary-value problem. Explicitly, the weighting $\epsilon: E_n \to k$ defined by $\epsilon(e_S)=(-1)^{|S|}$ for $S \subseteq \mathcal{P}_n$ and $\epsilon(f)=0$ for every $f \notin T_n$ is a desired solution. We immediately deduce the following result:

\begin{corollary} \label{cor:squarefree-cyclic-projective}
For every squarefree positive integer $n$, we have a normalization idempotent
\[
\Phi_n = \sum_{S \subseteq \mathcal P_n} (-1)^{|S|}e_S = \prod_{p \in \mathcal{P}_n}(1-e_{{p}}),
\]
such that $P_n\Phi_n \cong \Delta_n$.
\end{corollary}

\begin{remark}\label{rem:squarefree-necessity}
The restriction to cyclic groups of squarefree order is essential for the projectivity phenomenon above. Indeed, consider the larger category of all finite cyclic groups and let
\[
C_{p^r} \twoheadrightarrow C_{p^{r-1}}, \qquad r\geqslant 2,
\]
be the canonical quotient, where $p$ is prime. Identifying $\End(C_{p^r}) \cong \mathbb Z/p^r\mathbb Z$. The normalization fiber containing the identity is
\[
F = \{a \in \mathbb Z/p^r\mathbb Z \mid a \equiv 1\pmod{p^{r-1}}\}.
\]
Every element of $F$ is a unit modulo $p^r$, and hence every element of this fiber is an automorphism of $C_{p^r}$. By Corollary \ref{cor:identity-fiber-obstruction}, the standard module attached to $C_{p^r}$ is therefore not projective.
\end{remark}

We next compute Hom-spaces between these standard modules. Since $\Delta_n \cong P_n \Phi_n$, Yoneda's lemma gives
\[
\Hom_{k\C} (\Delta_n,\Delta_m) \cong \im\left( \Delta_m(\Phi_n): \Delta_m(C_n) \to \Delta_m(C_n) \right).
\]
By definition, $\Delta_m(C_n)$ is the quotient of $k\C(C_m,C_n)$ by the submodule spanned by the noninjective homomorphisms. Hence it has a natural basis
\[
\{\bar{f} \mid f: C_m \hookrightarrow C_n\},
\]
where $\bar f$ denotes the class of an injective homomorphism $f$. Under the above Yoneda identification, the corresponding elements of the Hom-space are represented by the vectors $\Phi_n\bar{f}$.

There are no injections $C_m\hookrightarrow C_n$ unless $m$ divides $n$, so $\Delta_m(C_n)=0$ otherwise. Suppose that $m\mid n$ and $m\neq n$. Choose a prime $p$ dividing $n/m$. The cyclic group $C_n$ has a unique subgroup of order $m$, and under the Chinese remainder decomposition this subgroup has zero $p$-coordinate. Hence the idempotent $e_{{p}}$ acts as the identity on the image of every injection $f: C_m \hookrightarrow C_n$. Therefore $(1-e_{\{p\}}) \bar{f} = 0$. Since $1-e_{{p}}$ is a factor of $\Phi_n$, it follows that $\Phi_n \bar{f} = 0$ for every basis element $\bar f$ of $\Delta_m(C_n)$. Consequently,
\[
\Hom_{k\C} (\Delta_n, \Delta_m) = 0.
\]

We therefore obtain the following result:
\begin{lemma} \label{lem:squarefree-cyclic-hom}
For squarefree integers $m, n$,
\[
\Hom_{k \C} (\Delta_n, \Delta_m) \cong
\begin{cases}
kG_n, & m=n,\\
0, & m \neq n.
\end{cases}
\]
\end{lemma}

Combining Corollary \ref{cor:squarefree-cyclic-projective} and Lemma \ref{lem:squarefree-cyclic-hom}, we know that standard modules form a hom-orthogonal family of projective generators. Morita theory then gives the
following Dold--Kan--type equivalence.

\begin{proposition}
Let $\C$ be the category of finite cyclic groups of squarefree order. Then one has
\[
\C \Mod \simeq \prod_{\substack{n \geqslant 1\\ n\ {\rm squarefree}}} k(\mathbb Z/n\mathbb Z)^\times \Mod.
\]
\end{proposition}

\begin{remark}
The category $\C$ is an inverse category in the sense of inverse-semigroup theory: every morphism $f$ admits a unique generalized inverse $f^\dagger$ satisfying $ff^\dagger f = f$ and $f^\dagger ff^\dagger = f^\dagger$. The equivalence above is therefore closely related to the general representation theory of inverse categories. Building on the M\"{o}bius-theoretic methods of Steinberg for inverse semigroups \cite{SteinbergMobius, SteinbergMobiusII}, Linckelmann proved that the category algebra of a finite inverse category decomposes as a product of matrix algebras over group algebras \cite{LinckelmannInverse}, and hence yields the corresponding finite version of the above equivalence. More recently, Alves and Velasco study Morita equivalences for inverse categories and use Kan-extension methods to analyze certain infinite cases \cite{AlvesVelasco}. Our argument is different in nature: it derives the equivalence directly from the generalized Reedy structure and the normalization graph, where the relevant M\"{o}bius-type idempotents appear as alternating balanced weightings on Boolean cubes.
\end{remark}

\subsection{The category $\mathrm{VA}_q$} \label{subsec:VAq}

Let $\C=\mathrm{VA}_q$ be the category of finite-dimensional vector spaces over $\mathbb F_q$, and write $V_n = \mathbb F_q^n$. Throughout this subsection, assume that $q$ is invertible in $k$. Our purpose is to reinterpret Kov\'{a}cs' construction in \cite{Kovacs} as a solution of the normalization equations. The projectivity of the standard modules, together with the role of the classical Kov\'{a}cs--Kuhn idempotents \cite{Kovacs,Kuhn} as normalization operators, will then follow from Theorem~\ref{thm:main} and Theorem~\ref{thm:normalization-correspondence}.

A complete set of representative unfactorizable morphisms in $\C^-$ with source $V_n$ is given by the quotient maps $\pi_L: V_n \twoheadrightarrow V_n/L$, where $L$ runs through the one-dimensional subspaces of $V_n$. The black vertices of the normalization graph are
\[
E_n = \End_{\mathbb F_q}(V_n)=M_n(\mathbb F_q),
\]
while its boundary is $G_n = \GL_n(\mathbb F_q)$.

Normalization fibers can be explicitly described as follows. Let $ L\subseteq V_n$ be a line and let $h:V_n \to V_n/L$ be a linear map. If
\[
F_{L,h} = \{A\in M_n(\mathbb F_q)\mid \pi_LA=h\}
\]
is nonempty and $A_0\in F_{L,h}$, then
\[
F_{L,h} = A_0+\Hom_{\mathbb F_q}(V_n,L).
\]
Indeed, for $A\in M_n(\mathbb F_q)$, one has $\pi_L A = \pi_LA_0$ if and only if $\im(A-A_0) \subseteq L$, and if and only if $A - A_0$ is contained in $\Hom_{\mathbb F_q}(V_n, L)$.

We seek a normalized weighting $\epsilon:M_n(\mathbb F_q)\to k$, or equivalently, the corresponding normalization idempotent
\[
\Phi_n = \sum_{A \in M_n (\mathbb F_q)}\epsilon(A)[A]
\]
such that $\pi_L \Phi_n = 0$ for every line $L \subseteq V_n$. Since the prescribed boundary coefficients of $\Phi_n$ are $1$ at the identity and $0$ at every other invertible matrix, it is natural to write $\Phi_n = 1-e_n$, where $e_n$ is supported on the singular matrices. Thus the problem becomes that of finding a solution $e_n \in k[\Sing_n(\mathbb F_q)]$ for the following system of linear equations:
\begin{equation}\label{eq:vaq-pi-e}
\pi_Le_n=\pi_L \qquad \text{for every line }L.
\end{equation}

Kov\'{a}cs' theorem \cite{Kovacs} provides an identity element $e_n \in k[\Sing_n(\mathbb F_q)]$ for the ideal spanned by the singular matrices. We claim that $e_n$ solves the system \eqref{eq:vaq-pi-e}. Indeed, fix a line $L\subseteq V_n$ and choose a linear section $j_L: V_n/L \to V_n$ of the quotient map $\pi_L$. Set $p_L = j_L\pi_L \in \End_{\mathbb F_q}(V_n)$. Then $p_L$ is singular and $\pi_L p_L = \pi_L$. Since $e_n$ is the identity of the singular matrix ideal, one has $p_L e_n = p_L$. Consequently,
\[
\pi_L e_n = \pi_Lp_Le_n = \pi_Lp_L = \pi_L.
\]
Thus $e_n$ satisfies $\pi_Le_n=\pi_L$ for every line $L\subseteq V_n$, which is precisely the system of normalization equations above. Therefore $\Phi_n = 1 - e_n$ satisfies $\pi_L\Phi_n=0$ for every $L$. Since $e_n$ is supported on the singular matrices, the coefficient function of $\Phi_n$ has boundary value $\delta_{1_{V_n}}$. Hence $\Phi_n$ corresponds to a normalized weighting on the normalization graph. By Theorems \ref{thm:main} and \ref{thm:normalization-correspondence}, the standard module $\Delta_n$ is projective and $P_n \Phi_n \cong \Delta_n$.

From this viewpoint, Kov'{a}cs' identity for the singular matrix ideal provides directly a solution of the normalization boundary-value problem, while the complementary idempotent $1-e_n$ used by Kuhn is precisely the associated normalization idempotent.

\subsection{Segal's category $\Gamma$}

Let $\Gamma$ be the category of finite pointed sets and basepoint-preserving maps. In \cite{Pirashvili}, Pirashvili defines the normalized part of a $\Gamma$-module by intersecting the kernels of the retractions $r_i$, and uses the associated commuting projections together with decompositions indexed by subsets of $[n]$. Our purpose here is to reinterpret this classical normalization in terms of the boundary value problem of the normalization graph. From this viewpoint, the subset decomposition becomes a Boolean subgraph of the normalization graph, and the usual inclusion--exclusion normalization is precisely its alternating balanced weighting.

Let us briefly recall the basic ingredients of the normalization appearing explicitly in Pirashvili's proof of the Dold--Kan--type theorem. For $1 \leqslant i \leqslant n$, let $\mathbf n_+ = \{0,1,\ldots,n\}$ with basepoint $0$, $\mathbf n_+^{(i)} = \mathbf n_+\setminus\{i\}$ with the inherited basepoint, and define $r_i: \mathbf n_+ \to \mathbf n_+^{(i)}$ by sending $i$ to $0$ and fixing every other element. Let $s_i: \mathbf n_+^{(i)} \to \mathbf n_+$ be the evident section, and $e_i=s_ir_i$ be the idempotent endomorphism in $E_n = \End_\Gamma(\mathbf n_+)$.

Let $J_n\subseteq P_n$ be the subfunctor generated by $r_1,\ldots,r_n$. Equivalently,
\[
J_n(X_+) = k \{f: \mathbf n_+ \to X_+ \mid f(i) = 0 \text{ for some } 1 \leqslant i \leqslant n\}.
\]
Define $\Delta_n^\Gamma = P_n/J_n$. Then
\begin{equation}\label{eq:gamma-standard-values}
\Delta_n^\Gamma(X_+) = k\{f:\{1,\ldots,n\}\to X_+\setminus\{0\}\}.
\end{equation}
Let $t^*(X_+)=k[X_+]/k[0]$ be the reduced linearization functor. Then the basis in \eqref{eq:gamma-standard-values} yields a canonical isomorphism $\Delta_n^\Gamma\cong(t^*)^{\otimes n}$, which are precisely the projective generators used by Pirashvili.

Note that the above projective generators are not the standard modules arising from the usual epi--mono Reedy structure on finite pointed sets. Thus we need to introduce the following variant of Theorem~\ref{thm:main} which applies to this slightly more general setting.

\begin{lemma}\label{lem:relative-normalization}
Let $x$ be an object of a small category $\C$. Suppose that a subfunctor $J \subseteq k\C(x, -)$ is generated by a family of morphisms $u_\alpha: x \to x_\alpha$, and $J(x)$ is the free $k$-submodule spanned by a subset $D \subseteq E_x = \End_\C(x)$. Then $k\C(x, -)/J$ is projective if and only if there exists a weighting $\epsilon: E_x \to k$
such that
\[
\epsilon(1_x) = 1, \qquad \epsilon(f)=0, \quad \forall \, f \in E_x \setminus (D \cup \{ 1_x\})
\]
and
\[
\sum_{\substack{g \in E_x\\u_\alpha g=h}}\epsilon(g)=0
\]
for every $\alpha$ and every morphism $h$ occurring as a value of $u_\alpha g$.
\end{lemma}

\begin{proof}
This can be proved by an argument similar to that of Theorem \ref{thm:main}. Given such a weighting $\epsilon$, define
\[
\Phi = \sum_{g \in E_x} \epsilon(g)g.
\]
The fiber equations are equivalent to $u_\alpha\Phi=0$ for every generator $u_\alpha$. Hence the endomorphism of $P_x = k\C(x, -)$ defined by precomposition with $\Phi$ annihilates $J$, and therefore factors through $P_x/J$. The boundary condition says precisely that the image of $\Phi$ in $(P_x/J)(x)$ is the class of $1_x$. Thus the induced map $P_x/J \to P_x$ is a section of the quotient map. Conversely, evaluating a section on the
class of $1_x$ produces such an element $\Phi$ and therefore such a weighting.
\end{proof}

For the quotient $\Delta_n^\Gamma = P_n/J_n$, the subset $D$ consists of those endomorphisms $f: \mathbf n_+ \to \mathbf n_+$ for which $f(i)=0$ for at least one $i\neq0$. Thus the surviving boundary is
\[
B_n = \{f \in E_n \mid f^{-1}(0)=\{0\}\}.
\]
For every subset $S \subseteq[n] = \{1,\ldots,n\}$, define $e_S \in E_n$ by
\[
e_S(j)=
\begin{cases}
0,&j\in S,\\
j,&j\notin S.
\end{cases}
\]
In particular, set $e_\varnothing = 1_{\mathbf n_+}$ and $e_i = e_{\{i\}}$. It is easy to check that $e_Se_T = e_{S\cup T}$. Moreover, by a similar argument in Subsection \ref{subsec:simplex}, one can show that
\[
r_ie_S = r_ie_T \, \Longleftrightarrow \, S \setminus\{i\} = T \setminus\{i\}.
\]
Consequently, every fiber of $(r_i)^{\ast}: E_n \to \Gamma(\mathbf n_+,\mathbf n_+^{(i)}) $ meets the set $\{e_S \mid S\subseteq[n]\}$ either trivially or in exactly two elements $\{e_S, e_{S\cup\{i\}}\}$ when $i$ is not contained in $S$. Thus the part of the normalization graph supported on the vertices $e_S$ is the Boolean $n$-cube. Its alternating weighting gives the desired normalization. Explicitly, the weighting $\epsilon: E_n \to k$ defined by $\epsilon(e_S) = (-1)^{|S|}$ for $S \subseteq [n]$ and $\epsilon(f) = 0$ for every endomorphism $f$ not of the form $e_S$ is normalized, and the associated normalization idempotent is
\[
\Phi_n = \sum_{S\subseteq[n]}(-1)^{|S|}e_S = \prod_{i=1}^n(1-e_i).
\]
By Lemma \ref{lem:relative-normalization}, $\Delta_n^\Gamma \cong(t^*)^{\otimes n}$ is projective, recovering Pirashvili's family of normalized projective generators.

\section{Applications}\label{sec:affine}

We now turn to a new family of applications of the normalization machinery, beginning with the category $\C = \AF_q$ of finite-dimensional affine spaces over $\mathbb F_q$ and all affine maps. The affine case leads not only to a new Dold--Kan--type equivalence, but also to a transfer principle for normalization idempotents, which in turn yields further equivalences for semilinear and semiaffine categories. To the best of our knowledge, these Dold--Kan--type equivalences have not previously appeared in the literature.

\subsection{The affine category}

Equip $\C = \AF_q$ with the usual epi--mono generalized Reedy structure, with degree given by dimension. For each $n \geqslant 0$, let $\mathbb A_n$ be the affine space underlying $V_n = \mathbb F_q^n$. Thus
\[
G_n=\Aut_{\C}(\mathbb A_n)=\AGL_n(q)
\]
and every affine endomorphism of $\mathbb A_n$ can be written uniquely as
\[
[A,b]: x \longmapsto Ax+b, \qquad A\in M_n(\mathbb F_q),\quad b\in V_n.
\]
Accordingly,
\[
E_n = \End_{\C}(\mathbb A_n) = M_n(\mathbb F_q) \ltimes V_n.
\]

As in the linear case, representatives of the unfactorizable morphisms in $\Cm$ with source $\mathbb A_n$ may be chosen to be the quotient maps $\pi_L: \mathbb A_n \twoheadrightarrow \mathbb A_n/L$, where $L$ runs through the one-dimensional linear subspaces of $V_n$. Indeed, every affine surjection becomes linear after postcomposition with a suitable translation of its codomain, and a linear surjection with one-dimensional kernel $L$ differs from $\pi_L$ only by an isomorphism of the codomain.

Fix a line $L\subseteq V_n$ and let $\pi_L: \mathbb A_n \twoheadrightarrow\mathbb A_n/L$ be the corresponding unfactorizable morphism in $\C^-$. For an affine map $h = [B,c]: \mathbb A_n \longrightarrow\mathbb A_n/L$, the corresponding normalization fiber is
\[
F_{L,h} = \{[A,b] \in E_n \mid \pi_L A = B,\ \pi_L b=c\}.
\]
Note that every affine map $h$ occurs in this way since $\pi_L$ is surjective. If $[A_0,b_0] \in F_{L,h}$, then
\[
F_{L,h} = \left\{ [A_0 + T, b_0 + \ell] \;\middle|\; T\in\Hom_{\mathbb F_q}(V_n,L),\ \ell\in L \right\}.
\]
Thus the normalization graph for $\mathbb A_n$ is obtained from the linear normalization fibers by adjoining one translation coordinate. In particular, $|F_{L,h}| = q^{n+1}$

Assume from now on that $q$ is invertible in $k$. Let $\epsilon_{\mathrm{lin}}: M_n(\mathbb F_q) \to k$ be the normalized weighting for $\mathrm{VA}_q$ constructed in Subsection~\ref{subsec:VAq}, and write
\[
\Phi_n^{\mathrm{lin}} = \sum_{A\in M_n(\mathbb F_q)} \epsilon_{\mathrm{lin}}(A)[A].
\]
We extend this weighting to affine endomorphisms by setting
\begin{equation}\label{eq:affine-weighting}
\epsilon_{\mathrm{aff}}([A,b])
=
\begin{cases}
\epsilon_{\mathrm{lin}}(A),&b=0,\\
0,&b\neq0.
\end{cases}
\end{equation}

\begin{proposition}\label{prop:affine-projective}
The weighting $\epsilon_{\mathrm{aff}}$ is normalized on the normalization graph of $\mathbb A_n$. Consequently,
\[
\widetilde\Phi_n = \sum_{A\in M_n(\mathbb F_q)} \epsilon_{\mathrm{lin}}(A)[A,0]
\]
is a normalization idempotent satisfying $P_{\mathbb A_n}\widetilde\Phi_n \cong \Delta_{\mathbb A_n}$.
\end{proposition}

\begin{proof}
Consider a normalization fiber $F_{L,h}$ with $h=[B,c]$. We have two cases:
\begin{itemize}
\item If $c \neq 0$, then no element of $F_{L,h}$ having nonzero weight under \eqref{eq:affine-weighting} can occur, since every element in the support of $\epsilon_{\mathrm{aff}}$ has translation part zero. Hence
\[
\sum_{f\in F_{L,h}}\epsilon_{\mathrm{aff}}(f)=0.
\]

\item If $c = 0$, then the intersection of $F_{L,h}$ with the support of $\epsilon_{\mathrm{aff}}$ is $\{[A,0] \mid \pi_LA=B\}$. In this case,
\[
\sum_{f \in F_{L,h}}\epsilon_{\mathrm{aff}}(f) = \sum_{\substack{A\in M_n(\mathbb F_q)\\\pi_LA=B}} \epsilon_{\mathrm{lin}}(A) = 0,
\]
where the last equality is precisely the corresponding linear normalization equation for $\VA_q$.
\end{itemize}
Thus $\epsilon_{\mathrm{aff}}$ is balanced.

Note that an affine endomorphism $[A,b]$ is invertible precisely when $A \in \GL_n(\mathbb F_q)$. By our construction, $\epsilon_{\mathrm{aff}}([A,b]) = 0$ whenever $b \neq0$, while on the zero-translation subgroup $\GL_n(\mathbb F_q)$ it restricts to $\epsilon_{\mathrm{lin}}$. Thus $\epsilon_{\mathrm{aff}}(1_{\mathbb A_n}) = 1$ and it vanishes on every other element of $\AGL_n(q)$.

We have shown that $\epsilon_{\mathrm{aff}}$ is a normalized weighting, so the conclusions follow from Theorem \ref{thm:normalization-correspondence}.
\end{proof}

In other words, the linear slice
\[
M_n(\mathbb F_q) \cong \{[A,0]\mid A\in M_n(\mathbb F_q)\} \subseteq E_n
\]
already supports a solution of the affine boundary-value problem. Every affine normalization fiber either misses this slice or intersects it in exactly the corresponding normalization fiber for $\mathrm{VA}_q$. Thus the Kov\'{a}cs--Kuhn normalization for linear maps extends directly to affine maps by assigning coefficient zero to every endomorphism with nonzero translation part.

Now we compute Hom-spaces between normalized projectives.

\begin{proposition}\label{prop:affine-hom}
Assume that $q$ is invertible in $k$. For $m, n \geqslant0$, one has
\[
\Hom_{k\C}(\Delta_n,\Delta_m)
\cong
\begin{cases}
k\AGL_n(q), & m = n,\\
kX_n, & m = n-1,\\
0, & \text{otherwise},
\end{cases}
\]
where $X_n = \left\{ f: \mathbb A_{n-1} \hookrightarrow \mathbb A_n \ \middle|\ 0 \notin\im(f) \right\}$. In particular, $|X_n| = |\GL_n(\mathbb F_q)|$.
\end{proposition}

\begin{proof}
Since $\Delta_n \cong P_{\mathbb A_n} \widetilde{\Phi}_n$, Yoneda's lemma gives
\begin{equation}\label{eq:affine-hom-yoneda}
\Hom_{k\C} (\Delta_n, \Delta_m) \cong \im \left(\Delta_m(\widetilde{\Phi}_n): \Delta_m(\mathbb A_n) \longrightarrow \Delta_m(\mathbb A_n) \right).
\end{equation}
By definition,
\[
\Delta_m(\mathbb A_n) = k\C(\mathbb A_m, \mathbb A_n) \big/ \I_m(\mathbb A_n),
\]
where $\I_m(\mathbb A_n)$ is spanned by the noninjective affine maps. Hence there is a natural isomorphism
\[
\Delta_m(\mathbb A_n) \cong k\Inj_{\mathrm{aff}}(\mathbb A_m, \mathbb A_n),
\]
and every nonzero basis class of $\Delta_m(\mathbb A_n)$ is represented by a unique affine injection $f: \mathbb A_m \hookrightarrow \mathbb A_n$. In particular, $\Delta_m(\mathbb A_n) = 0$ if $m > n$.

Assume $m \leqslant n$ and let $\bar{f}$ be such a basis class, represented by $f(x)=Bx+c$, where $B:V_m\to V_n$ is injective. Let $L_f =\im(B) + \mathbb{F}_qc$. Then
\[
\dim_{\mathbb{F}_q} L_f=
\begin{cases}
m, & 0 \in \im(f),\\
m+1, & 0 \notin \im(f).
\end{cases}
\]
Let $e_n^{\mathrm{lin}}$ denote the identity of the singular matrix ideal in $kM_n(\mathbb F_q)$, and let $\widetilde{e}_n$ be its image under the algebra embedding
\[
kM_n(\mathbb F_q) \longrightarrow kE_n, \qquad [A] \longmapsto [A,0].
\]
Thus $\widetilde{\Phi}_n = 1 - \widetilde{e}_n$. In particular, for every singular linear endomorphism $p$ of $V_n$,
\begin{equation}\label{eq:affine-phi-kills-singular}
\widetilde{\Phi}_n [p,0] = 0.
\end{equation}

We prove the conclusion case by case.

\medskip
\noindent
\textbf{Case 1: $m \leqslant n-2$.} Then $\dim_{\mathbb{F}_q} L_f\leqslant n-1$. Choose a linear projection $p: V_n \to V_n$ onto a proper subspace containing $L_f$ and acting as the identity on $L_f$. Thus $p$ is singular and $[p,0] \circ f=f$. Using \eqref{eq:affine-phi-kills-singular}, we obtain
\[
\widetilde{\Phi}_n f = \widetilde{\Phi}_n[p,0] \circ f = 0.
\]
It follows that $\Delta_m (\widetilde{\Phi}_n) (\bar{f}) = 0$. Hence \eqref{eq:affine-hom-yoneda} gives
\[
\Hom_{k\C}(\Delta_n, \Delta_m) = 0.
\]

\medskip
\noindent
\textbf{Case 2: $m=n-1$.} Define
\begin{align*}
X_n & = \{f: \mathbb A_{n-1} \hookrightarrow \mathbb A_n \mid 0 \notin \im(f)\}, \quad \bar{X}_n = \{\bar{f} \mid f \in X_n \};\\
Y_n & = \{f: \mathbb A_{n-1} \hookrightarrow \mathbb A_n \mid 0 \in \im(f)\}, \quad \bar{Y}_n = \{\bar{f} \mid f \in Y_n \}.
\end{align*}
Then $\Delta_{n-1}(\mathbb A_n) = k\bar{X}_n \oplus k\bar{Y}_n$. Moreover, if $f \in Y_n$, then $\dim_{\mathbb{F}_q} L_f = n-1$, and the same argument shows $\widetilde{\Phi}_n f = 0$. Thus $\widetilde{\Phi}_n$ kills $kY_n$.

Now choose an arbitrary $f \in X_n$ as well as a singular matrix $A$ such that $[A \, 0]$ occurs in the support of $\widetilde{\Phi}_n$. Consider the affine map
\[
[A,0] \circ f: x \longmapsto ABx + Ac.
\]
We have two cases:
\begin{itemize}
\item If it is not injective, then the term $\overline{[A \, 0] \circ f}$ is 0, and hence contributes nothing to $\widetilde{\Phi}_n \bar{f}$.

\item If it is injective, then $\rank(AB) = n-1$. Moreover, since 0 is not contained in the image of $f$, we know that $c \notin \im(B)$, so $[B \; c]$ is contained in $\GL_n(\mathbb{F}_q)$. Since $[AB\;Ac] = A[B\;c]$ and $[B \; c]$ is invertible, it follows that $\rank[AB\;Ac] = \rank A$. But the matrix $A$ is singular, while $AB$ already has rank $n-1$. This forces $\rank A = n-1$ and necessarily $Ac \in \im(AB)$. Therefore
\[
\im([A,0] \circ f) = Ac+\im(AB) = \im(AB),
\]
which contains $0$. It follows that $[A \, 0] \circ f$ is contained in $Y_n$.
\end{itemize}

By the above analysis, every term of $\widetilde{\Phi}_n \bar{f}$ arising from a singular matrix appearing in $\widetilde{\Phi}_n$ is either zero in $\Delta_{n-1}(\mathbb A_n)$ or is represented by an element of $Y_n$. Since the identity term of $\widetilde\Phi_n$ contributes $\bar f$, there exists $y_f \in kY_n$ such that
\[
\Delta_{n-1} (\widetilde\Phi_n) (\bar f) = \bar f+\bar y_f.
\]
Moreover, $\widetilde\Phi_n$ annihilates $kY_n$, and hence $\Delta_{n-1}(\widetilde\Phi_n)$ annihilates $k\bar Y_n = k\{\bar{f} \mid f \in Y_n\}$. Since
\[
\Delta_{n-1}(\mathbb A_n) = k\bar X_n\oplus k\bar Y_n,
\]
it follows that
\[
\im \Delta_{n-1} (\widetilde\Phi_n) = \Delta_{n-1} (\widetilde\Phi_n) (k\bar{X}_n \oplus k\bar{Y}_n) = \Delta_{n-1} (\widetilde\Phi_n) (k\bar{X}_n)  = k\{\bar f + \bar y_f \mid f \in X_n\}.
\]
These elements are linearly independent: under the projection
\[
k\bar X_n \oplus k\bar Y_n \longrightarrow k\bar X_n,
\]
the element $\bar f+\bar y_f$ maps to the basis element $\bar f$. Thus
\[
\Hom_{k\C}(\Delta_n,\Delta_{n-1}) \cong \im \Delta_{n-1}(\widetilde\Phi_n) \cong k\bar{X}_n \cong kX_n
\]
as claimed.

\medskip
\noindent
\textbf{Case 3: $m=n$.} This follows from \cite[Proposition 3.1]{LiHom}.

\medskip

Finally, for $n \geqslant 1$, an element of $X_n$ has the form $f(x) = Bx+c$ with $\rank B = n-1$ and $c \notin \im(B)$. Thus we obtain a bijection
\[
X_n \longrightarrow\GL_n(\mathbb F_q), \quad f \longmapsto [B\;c].
\]
The equality on cardinalities follows.
\end{proof}

\begin{remark}\label{rem:affine-asymmetry}
The nonzero adjacent Hom-space in Proposition~\ref{prop:affine-hom} may be viewed as a manifestation of the asymmetry between the positive and negative parts of the affine Reedy category. For finite-dimensional vector spaces, injections and surjections occur in exactly the same number. This symmetry is closely related to linear duality, which exchanges injections and surjections.

For affine spaces the situation is different. An affine injection $\mathbb A_m \hookrightarrow \mathbb A_n$ consists of an injective linear map together with a translation vector in $V_n$, whereas an affine surjection $
\mathbb A_n \twoheadrightarrow \mathbb A_m$ consists of a surjective linear map together with a translation vector in $V_m$. It is clear that in general there are more affine injections.

This numerical asymmetry does not by itself imply the existence of nonzero morphisms in adjacent degrees, but it reflects the geometric source of that phenomenon. After choosing an origin in $\mathbb A_n$, affine hyperplane embeddings split into those whose image contains the origin and those whose image does not. The former are annihilated by the normalization idempotent, just as in the linear case, whereas the latter survive and form a basis of the adjacent Hom-space. Thus the translation part breaks the positive--negative symmetry present in $\VA_q$ and leaves precisely one nontrivial off-diagonal layer after normalization.
\end{remark}

Combining the above propositions, we obtain the following Dold-Kan equivalence for the category of finite affine spaces.

\begin{theorem} \label{thm:affine-dold-kan}
Let $\C = \AF_q$, and assume that $q$ is invertible in $k$. Let $\mathscr{K}_{\mathrm{af}}$ be the $k$-linear category with objects $n \geqslant 0$ and
\[
\mathscr{K}_{\mathrm{af}} (m,n) = \Hom_{k\C} (\Delta_n, \Delta_m),
\]
with composition induced by composition of homomorphisms between standard modules. Then
\[
\C \Mod \simeq \mathscr{K}_{\mathrm{af}} \Mod.
\]
Moreover, as $k$-modules,
\[
\mathscr{K}_{\mathrm{af}}(m,n) \cong
\begin{cases}
k\AGL_n(q), & m=n,\\
kX_n, & m=n-1,\\
0, & \text{otherwise},
\end{cases}
\]
where
\[
X_n = \{f: A_{n-1} \hookrightarrow A_n \mid 0 \notin \im(f)\}.
\]
In particular, every composite of two consecutive off-diagonal morphisms is zero.
\end{theorem}

\begin{proof}
The conclusion follows from Lemma \ref{lem:projective generators}, and Propositions \ref{prop:affine-projective} and \ref{prop:affine-hom}.
\end{proof}

\subsection{The transfer principle}

In this subsection we introduce a transfer principle, which tells us that, in favorable situations, normalization need not be reconstructed in a larger category but can be inherited directly from a suitable wide subcategory. We use this principle to deduce Dold-Kan equivalences for the semilinear category and the semiaffine category.

Let $\D$ be a wide generalized Reedy subcategory of a generalized Reedy category $\C$, with compatible positive and negative subcategories. Suppose that for every pair of objects $x,y \in \C$ one has
\[
\C^-(x,y) = G_y\,\D^-(x,y) = \{ \sigma f \mid \sigma \in G_y = \Aut_{\C}(y), \, f \in \D^-(x, y) \},
\]
Since $\D$ is a subcategory of $\C$, it follows that $k\End_{\D}(x)$ is a subalgebra of $k\End_{\C}(x)$ for $x \in \Ob(\C)$.

\begin{proposition}\label{prop:transfer-normalization}
Under the above assumption, if $\{\Phi_x\}_{x \in \Ob(\C)}$ is a family of normalization idempotents for $\D$, then regarded as elements of $k\End_{\C}(x)$, it is also a family of normalization idempotents for $\C$. In particular, every standard module of $\C$ is projective.
\end{proposition}

\begin{proof}
Fix $x \in \Ob(\D)$ and write $\Phi = \Phi_x$. Since $k\End_{\D}(x) \subseteq k\End_{\C}(x)$, $\Phi$ remains an idempotent in $k\End_{\C}(x)$. Let $\I_x^{\D}$ and $\I_x^{\C}$ be the standard ideals in the two categories. Since $\Phi$ is a normalization idempotent for $\D$, one has $\I_x^{\D}=P_x^{\D}(1_x-\Phi)$. In particular, $1_x - \Phi$ is contained in $\I_x^{\D}(x)$, and hence is contained in $\I_x^{\C}(x)$. Thus $P_x^{\C} (1_x-\Phi) \subseteq \I_x^{\C}$.

Conversely, let $u: x \to y$ be a noninvertible morphism in $\C^-$. By assumption, $u=\sigma v$ for some $\sigma \in G_y$ and $v \in \D^-(x,y)$. Since $\sigma$ is invertible, $v$ is noninvertible. As $\Phi$ is a normalization idempotent for $\D$, $v\Phi=0$. Hence $u\Phi = \sigma v\Phi = 0$, and therefore $u=u(1_x-\Phi)$. Since $\I_x^{\C}$ is generated by the noninvertible morphisms in $\C^-$, this gives $\I_x^{\C} \subseteq P_x^{\C}(1_x-\Phi)$. Consequently, $\I_x^{\C} = P_x^{\C}(1_x-\Phi)$. Thus we obtain the following decomposition
\[
P_x^{\C} = P_x^{\C}\Phi \oplus P_x^{\C}(1_x-\Phi) = P_x^{\C}\Phi \oplus \I_x^{\C}
\]
and hence
\[
\Delta_x^{\C} = P_x^{\C}/\I_x^{\C} \cong P_x^{\C}\Phi.
\]
Therefore, $\Phi$ is a normalization idempotent for $\C$.
\end{proof}

\begin{example} \label{examples}
The preceding proposition applies in several natural situations.

\medskip
\noindent
\textbf{(1) Affine spaces from vector spaces.}
Let $\D=\VA_q$, and let $\C = \AF_q$ be the category of finite-dimensional affine spaces and affine maps, after choosing the standard origin on each object. Then $\D$ is a wide subcategory of $\C$. Every affine surjection
\[
f(x) = Ax + b: \mathbb A_n \twoheadrightarrow \mathbb A_m
\]
becomes linear after postcomposition with the translation $y\mapsto y-b$ of the codomain, so the assumption in the above proposition holds. Therefore the Kov\'{a}cs--Kuhn normalization idempotents for $\VA_q$ also serve as normalization idempotents for $\C$. This gives another proof of Proposition~\ref{prop:affine-projective}.

\medskip
\noindent
\textbf{(2) Betweenness maps from order-preserving maps.}
Let $\D$ be the simplex category, and let $\C$ be the category of finite linearly ordered sets and betweenness-preserving maps. Every surjective betweenness-preserving map is either order-preserving or order-reversing. In the latter case, postcomposition with the order-reversing automorphism of the codomain makes it order-preserving. Thus the classical simplicial normalization idempotents also serve as normalization idempotents for $\C$. In particular, the standard modules of $\C$ are projective over every commutative coefficient ring.

\medskip
\noindent
\textbf{(3) Semilinear maps from linear maps.}
Let $\D=\VA_q$, and let $\C=\SVA_q$ be the category with the same objects but with semilinear maps, regarded as maps rather than as pairs consisting of a map and a field automorphism. Every semilinear map factors through its image as a surjection followed by an injection, and its image is an $\mathbb F_q$-linear subspace. It is also routine to check that these factorizations are unique up to isomorphism of the intermediate image. Thus $\C$ is a generalized Reedy category. For $n \geqslant 1$, its automorphism group at $V_n$ is the semilinear group $\Gamma L_n(q)$, while $\Aut_{\C}(V_0)$ is trivial.

Let $f:V_n\twoheadrightarrow V_m$ be a semilinear surjection. If $m>0$, then $f$ is nonzero and hence has a uniquely determined associated field automorphism $\sigma\in\Aut(\mathbb F_q)$. Postcomposition with a semilinear automorphism of $V_m$ inducing $\sigma^{-1}$ turns $f$ into a linear surjection. If $m=0$, then $f$ is already the unique linear map to $V_0$. Hence
\[
\Cm(V_n,V_m) = \Aut_{\C}(V_m)\,\D^-(V_n,V_m)
\]
for all $m,n$. It follows from Proposition~\ref{prop:transfer-normalization} that, whenever $q$ is invertible in $k$, the Kov\'{a}cs--Kuhn normalization idempotents also give normalization idempotents for the semilinear category.

\medskip
\noindent
\textbf{(4) Semiaffine maps from affine maps.}
Similarly, let $\D=\AF_q$ and enlarge it to the category $\C=\SAF_q$ of finite-dimensional affine spaces and semiaffine maps, again regarded simply as maps. Every semiaffine map factors through its image as a surjection followed by an injection, and the usual image factorization makes $\C$ into a generalized Reedy category.

Let $f: A_n \twoheadrightarrow A_m$ be a semiaffine surjection. If $m>0$, then its associated field automorphism is uniquely determined, and postcomposition with a suitable semiaffine automorphism of $A_m$ makes $f$ affine. If $m=0$, then $f$ is the unique map to $A_0$ and is already affine. Consequently,
\[
\Cm(A_n,A_m) = \Aut_{\C}(A_m) \,\D^-(A_n,A_m)
\]
for all $m,n$. Therefore Proposition~\ref{prop:transfer-normalization} applies, and the affine normalization idempotents transfer to the semiaffine category.

\end{example}

We are ready to deduce Dold-Kan equivalences for the semilinear category $\SVA_q$ and the semiaffine category $\SAF_q$.

\begin{corollary}\label{cor:semilinear-semiaffine}
Assume that $q$ is invertible in $k$.

\begin{enumerate}
\item Let $\C = \SVA_q$ be the category of finite-dimensional $\mathbb F_q$-vector spaces and semilinear maps. Then
\[
\C \Mod \simeq \prod_{n \geqslant 0} k\Gamma L_n(q)\Mod,
\]
where we set $\Gamma L_0(q) = A \Gamma L_0(q) = 1$ for convenience.

\item Let $\C = \SAF_q$ be the category of finite-dimensional affine spaces over $\mathbb F_q$ and semiaffine maps, and assume that $q$ is invertible in $k$. Let $\mathscr{K}_{\mathrm{saf}}$ be the $k$-linear category with objects $n \geqslant 0$ and
\[
\mathscr{K}_{\mathrm{saf}} (m,n) = \Hom_{k\C} (\Delta_n, \Delta_m),
\]
with composition induced by composition of homomorphisms between standard modules. Then
\[
\C \Mod \simeq \mathscr{K}_{\mathrm{saf}} \Mod.
\]
Moreover, as $k$-modules,
\[
\mathscr K_{\mathrm{saf}}(m, n) \cong
\begin{cases}
kA\Gamma L_n(q), & m=n,\\
kX_n, & m=n-1,\\
0, & \text{otherwise},
\end{cases}
\]
and $X_n = \left\{ f:\mathbb A_{n-1}\hookrightarrow\mathbb A_n \ \middle|\ f\text{ is semiaffine and } 0 \notin \im(f) \right\}$.
\end{enumerate}
\end{corollary}

\begin{proof}
By the preceding transfer proposition, the Kov\'{a}cs--Kuhn normalization idempotents for $\VA_q$ also serve as normalization idempotents for the semilinear category, and the affine normalization idempotents serve as normalization idempotents for the semiaffine category. Hence in both cases all standard modules are projective.

For the semilinear category, let $f: V_m \hookrightarrow V_n$ be a semilinear injection with $m < n$. Its image is an $\mathbb F_q$-linear subspace of $V_n$. Choose a singular linear projection $p: V_n \to V_n$ which restricts to the identity on $\im(f)$. Then $pf = f$, $\Phi_n p = 0$, and therefore $\Phi_n f = 0$. Thus the standard modules are Hom-orthogonal in different degrees. In degree $n$, the injective semilinear endomorphisms are precisely elements of $\Gamma L_n(q)$, so
\[
\Hom(\Delta_n,\Delta_m)\cong
\begin{cases}
k\Gamma L_n(q), & m=n,\\
0, & m\neq n.
\end{cases}
\]
Morita theory therefore gives
\[
\SVA_q \Mod \simeq \prod_{n \geqslant 0} k\Gamma L_n(q) \Mod.
\]

The semiaffine case is proved in exactly the same way as the affine case. Indeed, write a semiaffine injection as $f(x) = B\sigma(x) + c$, where $\sigma \in \Aut(\mathbb F_q)$ and $B$ is injective. Let $L_f = \im(B) + \mathbb F_qc$.
\begin{itemize}
\item If $m \leqslant n-2$, then $\dim_{\mathbb{F}_q} L_f \leqslant n-1$, so a singular linear projection fixing $L_f$ shows that the normalization idempotent annihilates $f$.

\item If $m = n-1$ and $0 \in \im(f)$, the same argument shows that $f$ is annihilated by the normalization idempotent. In the case that $0 \notin \im(f)$, one has $c \notin \im(B)$, and hence $[B \;c] \in \GL_n(\mathbb F_q)$. The rank argument used in Proposition \ref{prop:affine-hom} is unchanged: for a singular matrix $A$, whenever $[A,0] \circ f$ remains injective, its image contains $0$.
\end{itemize}

It follows that
\[
\Hom(\Delta_n, \Delta_m) \cong
\begin{cases}
kA\Gamma L_n(q), & m=n,\\
kX_n, & m=n-1,\\
0, & \text{otherwise}.
\end{cases}
\]
The asserted Morita equivalence then follows.
\end{proof}

\subsection{Uniformly continuous representations of universal transformation monoids}

We conclude by relating the preceding equivalences to uniformly continuous representations of infinite transformation monoids. This is motivated by the reconstruction results of \cite{LiRel}, where the linear case $\VA_q$ was considered.

Let
\[
\Omega = \varinjlim_n\mathbb F_q^n = \bigoplus_{i\geqslant1}\mathbb F_qe_i
\]
be the infinite-dimensional vector space over $\mathbb F_q$. We consider the following transformation monoids:
\begin{align*}
\mathscr E_{\mathrm{sl}} & = \{f: \Omega \to \Omega \mid f \text{ is semilinear}\},\\
\mathscr E_{\mathrm{af}} & = \{f: x \mapsto Ax+b \mid A \in \End_{\mathbb F_q}(\Omega), \ b \in \Omega\},\\
\mathscr E_{\mathrm{saf}} & = \{f: \Omega \to \Omega \mid f \text{ is semiaffine}\}.
\end{align*}
Thus an element of $\mathscr E_{\mathrm{saf}}$ has the form $x \mapsto A\sigma(x)+b$, where $\sigma \in \Aut(\mathbb F_q)$, $A$ is $\mathbb F_q$-linear, and $b \in \Omega$. We equip these monoids with the topology of pointwise convergence and its associated uniform structure. Their groups of units are respectively
\[
\Gamma L(\Omega), \qquad \AGL(\Omega), \qquad A\Gamma L(\Omega).
\]

Let $\mathscr E$ be one of $\mathscr{E}_{\mathrm{sl}}$, $\mathscr E_{\mathrm{af}}$, or $\mathscr E_{\mathrm{saf}}$. Recall that a representation of $\mathscr E$ on a $k$-module $V$ equipped with the discrete topology is called \emph{uniformly continuous} if the representation map
\[
\mathscr E \longrightarrow \End_k(V)
\]
is uniformly continuous with respect to the pointwise uniformity on $\End_k(V)$. Equivalently, for every $v\in V$ there exists a finite subset $A \subseteq \Omega$ such that
\[
f|_A = g|_A \qquad \Longrightarrow \qquad fv = gv
\]
for all $f, g \in\mathscr E$. We write $k\mathscr E\Mod^{\mathrm{uc}}$ for the category of uniformly continuous $k\mathscr E$-modules.

Now we apply the reconstruction result in \cite{LiRel} to investigate uniformly continuous representations of $\mathscr E$. Note that these monoids are universal for the corresponding finite categories. Indeed, every semilinear map between finite-dimensional subspaces of $\Omega$ extends to a semilinear endomorphism of $\Omega$. Similarly, every affine map between finite-dimensional affine subspaces extends to an affine endomorphism of $\Omega$, and the same assertion holds for semiaffine maps. For instance, if $f: a+U \to b+W$ is affine and $f(a+u) = b+L(u)$, one may extend $L: U \to W$ to a linear endomorphism $\widetilde L$ of $\Omega$ and define $\widetilde f(x) = b + \widetilde L(x-a)$. The semiaffine case is identical, with $L$ replaced by a semilinear map.

There is a minor difference from the linear argument in \cite{LiRel}. Two semilinear maps agreeing on a finite subset need not agree on its linear span, since their associated field automorphisms may be different. This causes no difficulty here: since $\mathbb F_q$ is finite, every finite-dimensional linear or affine subspace is finite, and such subspaces form a cofinal family among the finite subsets used to define the uniform structure. Thus uniformly continuous representations may be tested on finite-dimensional subspaces in the semilinear case and on finite-dimensional affine subspaces in the affine and semiaffine cases.

Applying \cite[Corollary~7.20 and Subsection~7.7]{LiRel}, we obtain the following consequence.

\begin{corollary}\label{cor:universal-transformation-monoids}
Let $k$ be a commutative ring. There are equivalences of categories:
\begin{align*}
k\mathscr E_{\mathrm{sl}} \Mod^{\mathrm{uc}} & \simeq \SVA_q\Mod,\\
k\mathscr E_{\mathrm{af}} \Mod^{\mathrm{uc}} & \simeq \AF_q\Mod,\\
k\mathscr E_{\mathrm{saf}} \Mod^{\mathrm{uc}} & \simeq \SAF_q\Mod.
\end{align*}

If $q$ is invertible in $k$, then we have:
\begin{align*}
k\mathscr E_{\mathrm{sl}}\Mod^{\mathrm{uc}} & \simeq \prod_{n\geqslant0}k\Gamma L_n(q)\Mod,\\
k\mathscr E_{\mathrm{af}}\Mod^{\mathrm{uc}} & \simeq \mathscr K_{\mathrm{af}}\Mod,\\
k\mathscr E_{\mathrm{saf}}\Mod^{\mathrm{uc}} & \simeq \mathscr K_{\mathrm{saf}}\Mod.
\end{align*}
\end{corollary}

\begin{proof}
For each of the three transformation monoids, the category obtained from restrictions to finite-dimensional linear or affine subspaces is precisely the corresponding category $\SVA_q$, $\AF_q$, or $\SAF_q$, by the extension property described above. Every inclusion between the relevant finite subspaces admits a retraction of the same type. Therefore \cite[Corollary~7.20]{LiRel} gives the first three equivalences.

When $q$ is invertible in $k$, the remaining equivalences follow from the Dold--Kan--type equivalences established above for the semilinear, affine, and semiaffine categories.
\end{proof}

\end{document}